\documentclass[10pt]{article}

\usepackage{lmodern}
\usepackage{amsmath}

\usepackage[T1]{fontenc}
\usepackage[utf8]{inputenc}
\usepackage{authblk}
\usepackage{amsfonts}
\usepackage{rotating}
\usepackage{amssymb}
\usepackage[english]{babel}
\usepackage{xcolor}
\usepackage{amsthm}
\usepackage{graphicx}
\usepackage{mathrsfs}
\usepackage{makecell}
\usepackage{microtype}
\usepackage{mathscinet}
\usepackage{array}
\usepackage{multirow}
\usepackage{enumerate}
\usepackage[cal=boondoxo,bb=ams]{mathalfa}
\usepackage{hyperref}

\usepackage{pgfplots} 
\usepackage{tkz-euclide}
\usepackage{color}
\usepackage{enumitem}
\usepackage[cal=boondoxo,bb=ams]{mathalfa}

\hypersetup{hidelinks}

\newtheorem{theorem}{Theorem}[section]
\newtheorem{prop}{Proposition}[section]
\newtheorem{lemma}{Lemma}[section]

\newtheorem{remark}{Remark}[section]
\newtheorem{exam}{Example}[section]

\newcommand{\ml}{\mathcal}
\newcommand{\mb}{\mathbb}

\def\XXint#1#2#3{{\setbox0=\hbox{$#1{#2#3}{\int}$ }
		\vcenter{\hbox{$#2#3$ }}\kern-.6\wd0}}

\title{$L^p-L^q$ estimates for the dissipative and conservative Moore-Gibson-Thompson equations}

\author[1]{Wenhui Chen\thanks{Wenhui Chen (wenhui.chen.math@gmail.com)}}
\affil[1]{School of Mathematics and Information Science, Guangzhou University, \authorcr 510006 Guangzhou, China}

\author[2]{Mengjun Ma\thanks{Mengjun Ma (mamj7@mail2.sysu.edu.cn)}}
\author[2]{Xulong Qin\thanks{Xulong Qin (qinxul@mail.sysu.edu.cn)}}
\affil[2]{School of Mathematics, Sun Yat-Sen University, 510275 Guangzhou, China}

\date{}

\begin{document}
	\maketitle

	\begin{abstract}
		\medskip
		This paper studies some $L^p-L^q$ estimates for the dissipative or conservative Moore-Gibson-Thompson (MGT) equations in the whole space $\mb{R}^n$. Our contributions are twofold. By applying the Fourier analysis associated with the modified Bessel function in the dissipative case, we derive some $L^p-L^q$ estimates of solutions. Then, introducing a good unknown related to the free wave equation in the conservative case, some $L^p-L^q$ estimates of solutions with the admissible closed triangle range of exponents are deduced. These results show some essential influences of dissipation from the MGT equations in the $L^q$ framework.
		\\
		
		
		\noindent\textbf{AMS Classification (2020)} 35B40, 	35G10, 42B37 
	\end{abstract}
	\fontsize{12}{15}
	\selectfont

	\section{Introduction}\label{Section_Introduction}\setcounter{equation}{0}
	\hspace{5mm}In recent years, the Moore-Gibson-Thompson (MGT) equations, named after the early works of F.K. Moore, W.E. Gibson \cite{MooreGibson1960} in 1960 and of P.A. Thompson \cite{Thompson1972} in 1972, as well as some related models have caught a lot of attention. They arise in the nonlinear acoustics to describe the propagation of sound in thermo-viscous/inviscid fluids (cf. \cite{Blackstock-1963,Hamilton-Blackstock-1998,Jordan-2014}) from some applications of medical and industrial uses of high-intensity ultra sound, for example, medical imaging and therapy, ultrasound cleaning and welding (cf. \cite{Abramov-1999,Dreyer-Krauss-Bauer-Ried-2000,Kaltenbacher-Landes-Hoffelner-Simkovics-2002}). The well-known linear MGT equations are realized through the third-order hyperbolic partial differential equations (PDEs)
	\begin{align*}
	\tau\varphi_{ttt}+\varphi_{tt}-\Delta\varphi-(\delta+\tau)\Delta\varphi_t=0
	\end{align*}
with the thermal relaxation $\tau>0$ (from the Cattaneo law of heat conduction) and the diffusivity of sound $\delta\geqslant0$ (from the Navier-Stokes equations if $\delta>0$ or the Euler equations if $\delta=0$), where the unknown function $\varphi=\varphi(t,x)\in\mb{R}$ is referred to the acoustic velocity potential in the classical theory of acoustic waves. Therefore, one may notice that the presence of (viscous) dissipation $-\delta\Delta\varphi_t$ will greatly influence on the sound waves propagation physically and some qualitative properties of solutions mathematically. In particular, there is a transition in the linear models that can be described by an exponential stable strongly continuous semigroup in the case $\delta>0$ to the limit case $\delta=0$, where the exponential stability of semigroup  is lost and it holds the conservation of suitable defined energy \cite{Kaltenbacher-Lasiecka-Marchand-2011,Marchand-McDevitt-Triggiani-2012,Chen-Palmieri=2020}. For this reason, $\delta>0$ and $\delta=0$ in the MGT equations are always referred to the dissipative case and, respectively, the conservative case, whose solutions are endowed with the superscripts, i.e. $\varphi^{>0}$ and $\varphi^{=0}$, for the sake of simplicity.
	
	In this paper, we mainly contribute to qualitative properties of $L^q$ solutions to the following MGT equations (cf. \cite{Pellicer-Said-Houari=2019,Chen-Ikehata=2021,Chen-Takeda=2023,Chen-Gong=2024} for $\delta>0$ and \cite{Chen-Palmieri=2020,Chen=2024,Chen-Ikehata=2025} for $\delta=0$):
	\begin{align}\label{Eq_MGT}
		\begin{cases}
			\tau\varphi_{ttt}+\varphi_{tt}-\Delta\varphi-(\delta+\tau)\Delta\varphi_t=0,&x\in\mb{R}^n,\ t>0,\\
			\varphi(0,x)=\varphi_0(x),\ \varphi_t(0,x)=\varphi_1(x),\ \varphi_{tt}(0,x)=\varphi_2(x),&x\in\mb{R}^n, 
		\end{cases}
	\end{align}
with $\tau>0$ and $\delta\geqslant0$.
Our aim is to clarify asymptotic behavior of solutions to the dissipative MGT equations and the conservative MGT equations, determined by the value of $\delta$, in the $L^q$ ($q\neq2$) framework. Particularly, in the general $L^q$ framework, the parabolic-like structure in the dissipative case $\delta>0$ will be presented in Theorem \ref{thm:L^m-L^q-estimate}, whereas the hyperbolic-like (precisely, the wave-like) structure in the conservative case $\delta=0$ will be shown in Theorem \ref{Thm-Conser-MGT}.

 In the following we address a brief review on the MGT equations in the whole space $\mb{R}^n$ (concerning the bounded domain case, we refer the interested reader to \cite{Kaltenbacher-Lasiecka-Marchand-2011,Marchand-McDevitt-Triggiani-2012,Kaltenbacher-Lasiecka-Pos-2012,Conejero-Lizama-Rodenas-2015,Dell-Pata=2017,Kaltenbacher-Nikolic-2019,B-L-2020,Kaltenbacher-Niko-2021,Niko-Winker=2024} and references given therein).

Let us focus on the linear dissipative MGT equation \eqref{Eq_MGT} with $\delta>0$, which was initially studied by \cite{Pellicer-Said-Houari=2019}. To be specific, the authors of \cite{Pellicer-Said-Houari=2019} employed energy methods in the Fourier space combined with suitable Lyapunov functionals to derive energy estimates, and eigenvalues expansions to investigate some $L^2$ estimates for the solution itself. Soon afterwards, by applying the explicit representation of solutions and the Fourier analysis, \cite{Chen-Ikehata=2021} obtained large time optimal $L^2$ estimates (optimal growth for $n\leqslant 2$ and decay for $n\geqslant 3$).
As a continuation, in \cite{Chen-Takeda=2023} the authors made use of asymptotic analysis and refined Fourier analysis to capture its optimal leading term and second-order large time profile, which are determined by the diffusion waves. Then, the recent paper \cite{Chen-Gong=2024} introduced a good energy unknown to derive sharp $L^p-L^{q}$ estimates with $1\leqslant p\leqslant 2\leqslant q\leqslant +\infty$ for the energy term. Turning to the small parameter limits, \cite{Chen=2024} stated local (in time) inviscid limits as $\delta\downarrow 0$, and \cite{Chen-Ikehata=2021,Chen-Gong=2024} discovered global (in time) singular limits as well as higher-order asymptotic expansions of solutions associated with a singular layer as $\tau\downarrow 0$ via the multi-scale analysis and suitable energy methods in the Fourier space. However, the only available results in the literature \cite{Chen=2024,Chen-Ikehata=2025} concerning the linear conservative MGT equation \eqref{Eq_MGT} with $\delta=0$ concentrate on large time optimal $L^2$ estimates (optimal growth for $n\leqslant 2$ and boundedness for $n\geqslant 3$). These known results imply the loss of decay properties (when $n\geqslant 3$) if one drops the dissipation $-\delta\Delta\varphi_t$, but the reservation of growth properties (when $n\leqslant 2$), in the $L^2$ framework.

We now turn our review to nonlinear MGT equations in $\mb{R}^n$. The Cauchy problem for the semilinear MGT equations was firstly studied by \cite{Chen-Palmieri=2020,Chen-Palmieri=2021} with power nonlinearities $|\varphi|^m$ or $|\varphi_t|^m$. By developing an iteration method with slicing procedure to deal with the unbounded exponential  multiplier, they obtained blow-up of energy solutions for the semilinear conservative MGT equations with non-negative and compactly supported data under some conditions on the power $m$ (i.e. the sub-Strauss case and the sub-Glassey case). Later, in the dissipative situation $\delta>0$, \cite{Chen-Ikehata=2021,Shi-Zhang-Cai-Liu=2022} rigorously justified existence results for small data global (in time) $L^2$ solutions by using sharp $(L^2\cap L^1)-L^2$ estimates, and finite time blow-up results for weak solutions by using classical test function methods, under suitable conditions on the power $m$. Another modern topic is the dissipative Jordan-MGT equation (see \cite{Jordan-2014} for its detailed derivation by the Lighthill scheme of approximation to the fully compressible Navier-Stokes-Cattaneo equations under irrotational flows) in the whole space $\mb{R}^n$ whose global (in time) well-posedness and $L^2$ growth/decay estimates results were derived by \cite{Racke-Said-2020,Said-Houari=Large-Norm,Chen-Takeda=2023} for the $H^s$ solutions, and \cite{Said-Houari=Besov} for the $B_{2,1}^s$ solutions (here, $B_{2,1}^s$ denotes the Besov space based on $L^2$). On the contrary, concerning the Cauchy problem for the conservative Jordan-MGT equation, the authors of \cite{Chen-Liu-Palmieri-Qin=2023} discovered that energy solutions blow up in finite time when $n\leqslant 3$ by assuming some suitable weighted assumptions on the Cauchy data, where the upper bounds of lifespan were also estimated. 

 To the best of our knowledge, the previous studies are heavily based on the $L^2$ framework (energy methods and the Plancherel theorem can be widely used in different settings). Nevertheless, the qualitative properties of solutions to dissipative or conservative MGT equations in the $L^q$ framework are generally open, even for their linearized models. It is well-known that the investigation for linear problems is not
 only significant for understanding some underlying physical phenomena but also crucial for studying some corresponding nonlinear problems. For these reasons, we will partly answer the above questions by studying the linear Cauchy problems \eqref{Eq_MGT} for $\delta>0$ and $\delta=0$, respectively, in this manuscript via Theorems \ref{thm:L^m-L^q-estimate} and \ref{Thm-Conser-MGT}.
	
	\paragraph{\large Notation} The generic constants $c,C$ may change from line to line but are independent of the time variable. We write $f\lesssim g$ if there exists a positive constant $C$ such that $f\leqslant Cg$. Basing on the Lebesgue space $L^p$, we denote by $H^s_p$ and $\dot{H}^s_p$, respectively, the Bessel potential space and the Riesz (or homogeneous Bessel) potential space with $s\in\mb{R}$. The differential operator $|D|^{s}$ has its symbol $|\xi|^s$ with $s\in\mb{R}$. The H\"older conjugate of $p$ is $p'$ such that $1/p+1/p'=1$. We denote by $\lfloor a\rfloor:=\max\{A\in\mb{Z}: A\leqslant a\in\mb{R}\}$ the floor function, and by $[a]_+:=\max\{a,0\}$ the positive part of $a\in\mb{R}$. 

	\section{Dissipative MGT equation in the $L^q$ framework}\label{Section_MGT}\setcounter{equation}{0}
\hspace{5mm}Let us point out that our study on the dissipative MGT equation \eqref{Eq_MGT} in the $L^q$ framework with $q\in(1,+\infty)$ is not simply a generalization of previous literature \cite{Pellicer-Said-Houari=2019,Chen-Ikehata=2021,Chen-Takeda=2023,Chen-Gong=2024}. 
\begin{itemize}
	\item Different from the classical works \cite{Pellicer-Said-Houari=2019,Chen-Ikehata=2021,Chen-Takeda=2023} in the $L^2$ framework, the Plancherel equality $\|f\|_{L^2}=\|\widehat{f}\|_{L^2}$ does not work anymore in the general $L^q$ framework.
	\item Different from the recent paper \cite{Chen-Gong=2024} in the $L^q$ framework with $q\in[2,+\infty)$,  the well-known inequality $\|f\|_{L^q}\leqslant\|\widehat{f}\|_{L^{q'}}$ does not hold when $q\in(1,2)$, where it highly restricted the initial data belonging to $L^p$ with $p\in[1,2]$.
	\item Different from the consideration in \cite{Chen-Gong=2024} for a suitable energy term containing $\varphi_t^{>0}$ and $|D|\varphi^{>0}$, it is interesting to understand qualitative properties for the solution itself instead of the last energy terms.
\end{itemize}
To sum up, the phase space analysis cannot be directly applied   in the general $L^q$ framework for our models. Motivated by \cite{Narazaki-Reissig=2013,Dao-Reissig=2019,Dao-Reissig=2019-02,Duong-Dao-Bui=2025} for second-order (in time) damped wave equations, we are going to apply the WKB analysis and the Fourier analysis associated with the modified Bessel function or the Bernstein theorem (cf. Appendix \ref{Appendix-A}) to the dissipative MGT equation \eqref{Eq_MGT}. It seems to be the first work on higher-order (in time) PDEs by these approaches.

\subsection{Representation of solution in the Fourier space}
\hspace{5mm}We as usual apply the partial Fourier transform with respect to the spatial variables to the linear Cauchy problem \eqref{Eq_MGT} with $\delta>0$ which yields 
\begin{align*}
	\begin{cases}
		\tau\widehat{\varphi}_{ttt}^{>0}+\widehat{\varphi}_{tt}^{>0}+(\delta+\tau)|\xi|^2\widehat{\varphi}_t^{>0}+|\xi|^2\widehat{\varphi}^{>0}=0,&\xi\in\mb{R}^n,\ t>0,\\
		\widehat{\varphi}^{>0}(0,\xi)=\widehat{\varphi}_0^{>0}(\xi),\ \widehat{\varphi}_t^{>0}(0,\xi)=\widehat{\varphi}_1^{>0}(\xi),\ \widehat{\varphi}_{tt}^{>0}(0,\xi)=\widehat{\varphi}_2^{>0}(\xi),&\xi\in\mb{R}^n.
	\end{cases}
\end{align*}
Its corresponding characteristic equation is given by the $|\xi|$-dependent cubic
\begin{align}\label{Eq_characteristic_root}
	\tau\lambda^3+\lambda^2+(\delta+\tau)|\xi|^2\lambda+|\xi|^2=0.
\end{align}
Then, the roots $\lambda_j=\lambda_j(|\xi|)$ with $j=1,2,3$ to \eqref{Eq_characteristic_root} can be expanded straightforwardly by Taylor-like expansions as follows:
\begin{itemize}
	\item for small frequencies  $|\xi|\ll 1$ (see \cite[Proposition 2.3]{Chen-Takeda=2023}),
	\begin{align*}
		\lambda_1(|\xi|)&=-\frac{1}{\tau}+\delta |\xi|^2+O(|\xi|^4),\\ \mu_{\mathrm{R}}(|\xi|)&=-\frac{\delta}{2}|\xi|^2-\frac{\tau\delta(\delta-\tau)}{2}|\xi|^4+O(|\xi|^6),\\ \mu_{\mathrm{I}}(|\xi|)&=|\xi|+\frac{\delta(4\tau-\delta)}{8}|\xi|^3+O(|\xi|^5);
	\end{align*}
	\item for large frequencies $|\xi|\gg1$ (we require higher-order expansions than those in \cite{Pellicer-Said-Houari=2019,Chen-Ikehata=2021,Chen-Takeda=2023} in order to estimate the remainders precisely),
	\begin{align*}
		\lambda_1(|\xi|)&=-\frac{1}{\delta+\tau}-\frac{\delta}{(\delta+\tau)^4}|\xi|^{-2}+O(|\xi|^{-4}),\\ \mu_{\mathrm{R}}(|\xi|)&=-\frac{\delta}{2\tau(\delta+\tau)}+\frac{\delta}{2(\delta+\tau)^4}|\xi|^{-2}+O(|\xi|^{-4}),\\ \mu_{\mathrm{I}}(|\xi|)&=\sqrt{\frac{\delta+\tau}{\tau}}|\xi|-\frac{\delta(\delta+4\tau)}{8\tau(\delta+\tau)^3}\sqrt{\frac{\delta+\tau}{\tau}}|\xi|^{-1}+O(|\xi|^{-3});
	\end{align*}
\end{itemize}
where $\lambda_{2,3}=\mu_{\mathrm{R}}\pm i\mu_{\mathrm{I}}$ are complex conjugate for small and large frequencies because the discriminant of the cubic \eqref{Eq_characteristic_root} is strictly negative in both cases.

 By the classical ordinary differential equations theory with the pairwise distinct characteristic roots, one may obtain the representation of solution in the Fourier space localizing in the small as well as large frequencies zones
\begin{align*}
	\notag\widehat{\varphi}^{>0}&=\left(\frac{-(\mu_{\mathrm{I}}^2+\mu_{\mathrm{R}}^2)}{\Lambda_0}\,\mathrm{e}^{\lambda_1 t}+\frac{2\mu_{\mathrm{R}}\lambda_1-\lambda_1^2}{\Lambda_0}\cos (\mu_{\mathrm{I}} t)\,\mathrm{e}^{\mu_{\mathrm{R}} t}+\frac{\lambda_1(\mu_{\mathrm{R}}\lambda_1+\mu_{\mathrm{I}}^2-\mu_{\mathrm{R}}^2)}{\mu_{\mathrm{I}}\Lambda_0}\sin (\mu_{\mathrm{I}} t)\,\mathrm{e}^{\mu_{\mathrm{R}} t}\right)\widehat{\varphi}_0^{>0}\\ \notag 
	&\quad+\left(\frac{2\mu_{\mathrm{R}}}{\Lambda_0}\,\mathrm{e}^{\lambda_1 t}+\frac{-2\mu_{\mathrm{R}}}{\Lambda_0}\cos (\mu_{\mathrm{I}} t)\,\mathrm{e}^{\mu_{\mathrm{R}} t}+\frac{\mu_{\mathrm{R}}^2-\mu_{\mathrm{I}}^2-\lambda_1^2}{\mu_{\mathrm{I}}\Lambda_0}\sin (\mu_{\mathrm{I}} t)\,\mathrm{e}^{\mu_{\mathrm{R}} t}\right)\widehat{\varphi}_1^{>0}\\ \notag 
	& \quad +\left(\frac{-1}{\Lambda_0}\,\mathrm{e}^{\lambda_1 t}+\frac{1}{\Lambda_0}\cos (\mu_{\mathrm{I}} t)\,\mathrm{e}^{\mu_{\mathrm{R}} t}+\frac{-(\mu_{\mathrm{R}}-\lambda_1)}{\mu_{\mathrm{I}}\Lambda_0}\sin (\mu_{\mathrm{I}} t)\,\mathrm{e}^{\mu_{\mathrm{R}} t}\right)\widehat{\varphi}_2^{>0}\\ 
	&:=\left(\widehat{K}_0^1+\widehat{K}_0^{\cos}+\widehat{K}_0^{\sin}\right)\widehat{\varphi}_0^{>0}+\left(\widehat{K}_1^1+\widehat{K}_1^{\cos}+\widehat{K}_1^{\sin}\right)\widehat{\varphi}_1^{>0}+\left(\widehat{K}_2^1+\widehat{K}_2^{\cos}+\widehat{K}_2^{\sin}\right)\widehat{\varphi}_2^{>0},
\end{align*}
where we set $\Lambda_0:=2\mu_{\mathrm{R}}\lambda_1-\mu_{\mathrm{I}}^2-\mu_{\mathrm{R}}^2-\lambda_1^2$ for simplicity. Furthermore, we denote the kernels for each initial data via
\begin{align}\label{widehat-K-expression}
	\widehat{K}_{\ell}:=\widehat{K}_{\ell}^1+\widehat{K}_{\ell}^{\cos}+\widehat{K}_{\ell}^{\sin} \ \mbox{with} \ \ell=0,1,2.
\end{align}
The last explicit representation is a reorganization of \cite[Equation (23)]{Chen-Takeda=2023} according to the features of Fourier multipliers.

\subsection{Preliminary on estimates for Fourier multipliers}
\hspace{5mm}We now introduce the radially symmetric as well as smooth cut-off functions $\chi_1(|\xi|)$, $\chi_2(|\xi|)$ and $\chi_3(|\xi|)$, respectively, by 
\begin{align*}
	\chi_1(|\xi|):=\begin{cases}
		1 & \mbox{if}\ |\xi|\leqslant \epsilon_0,\\ 
		0 &\mbox{if}\  |\xi|\geqslant 2\epsilon_0,\end{cases}\qquad \chi_3(|\xi|):=\begin{cases}
		1 &\mbox{if}\  |\xi|\geqslant 2N_0,\\ 
		0 &\mbox{if}\  |\xi|\leqslant N_0,\end{cases}
\end{align*}
and $\chi_2(|\xi|):=1-\chi_1(|\xi|)-\chi_3(|\xi|)$ with $0<\epsilon_0\ll 1$ as well as $N_0\gg 1$. We hereafter simply denote $\|f\|_{L^q_{\chi_k}}:=\|\chi_{k}(|D|)f\|_{L^q}$ with $k=1,2,3$.

Notice that $\mathrm{Re}\, \lambda_j(|\xi|)<0$ for any $j=1,2,3$  and $\{\xi\in\mb{R}^n: |\xi|\leqslant \epsilon_0\ \mbox{or}\ |\xi|\geqslant 2N_0\}$ from the last subsection. Thanks to the negative real parts of eigenvalues and the compactness of bounded frequency zone, we trivially claim an exponential stability for bounded frequencies.
\begin{prop}\label{prop:L^m-L^q-middle}
	Let $1\leqslant p\leqslant q\leqslant +\infty$ and $s\geqslant0$.	The solution localizing in the bounded frequency zone to the dissipative MGT equation \eqref{Eq_MGT} with $\delta>0$ satisfies the following $L^p-L^q$ estimates: 
	\begin{align*}
		\|\varphi^{>0}(t,\cdot)\|_{H^s_{q,\chi_2}}\lesssim \mathrm{e}^{-ct}\|(\varphi_0^{>0},\varphi_1^{>0},\varphi_2^{>0})\|_{(L^p)^3}.
	\end{align*}
\end{prop}

For this reason,  the remaining parts of this section study $L^p-L^q$ estimates for the solutions localizing in the small and large frequencies zones, respectively, in Subsections \ref{Sub-Section-Small} and \ref{Sub-Section-Large}. Then, by gluing all derived $L^p-L^q$ estimates in different zones, we in Subsection \ref{Sub-Section-Smuumary} conclude some refined $L^q$ estimates of solutions $\partial_t^j|D|^s\varphi^{>0}(t,\cdot)$ to the dissipative MGT equation for any $q\in(1,+\infty)$.

As preparations, we propose the next lemma in the $L^1$ norm to deal with several Fourier multipliers related to $\lambda_{2,3}$.
\begin{lemma}\label{lem-oscillating-sincos}
    Let $n\geqslant 1$ and $\beta\geqslant 0$. The following $L^1$ estimates hold:
	\begin{align*}
        \left\| \ml{F}^{-1}_{\xi\to x}\left(\chi_1(|\xi|)\,\mathrm{e}^{-c_1|\xi|^2t}\,|\xi|^{2\beta}\frac{\sin(c_2|\xi|t)}{c_2|\xi|}\right)\right\|_{L^1}&\lesssim (1+t)^{\frac{1}{2}(2+\lfloor\frac{n}{2}\rfloor)+\frac{1}{2}-\beta} ,\\ 
        \left\| \ml{F}^{-1}_{\xi\to x}\left(\chi_1(|\xi|)\,\mathrm{e}^{-c_1|\xi|^2t}\,|\xi|^{2\beta}g_0(|\xi|t)\right)\right\|_{L^1}&\lesssim \begin{cases}
        	(1+t)^{\frac{n}{4}-\beta}&\mbox{if}\ \beta\in\{0\}\cup(\frac{1}{2},+\infty),\\
        	(1+t)^{\frac{1}{2}(2+\lfloor\frac{n}{2}\rfloor)-\beta}&\mbox{if}\ \beta\in(0,\frac{1}{2}],
        \end{cases}
    \end{align*}
	where $g_0(|\xi|t)\in\{\sin(c_2|\xi|t),\cos(c_2|\xi|t)\}$ with $c_1>0$ and $c_2\in \mathbb{R}\backslash\{0\}$.
\end{lemma}
\begin{remark}In \cite[Lemma 3.2 with $\delta=\sigma=1$]{Dao-Reissig=2019}, they got the estimate
\begin{align*}
        \left\| \ml{F}^{-1}_{\xi\to x}\left(\chi_1(|\xi|)\,\mathrm{e}^{-c_1|\xi|^2t}\,|\xi|^{2\beta}\cos(c_2|\xi|t)\right)\right\|_{L^1}\lesssim (1+t)^{\frac{1}{2}(2+\lfloor\frac{n}{2}\rfloor)-\beta},
\end{align*}
but we partly improve their result when $\beta\in\{0\}\cup(\frac{1}{2},+\infty)$, i.e. in Lemma \ref{lem-oscillating-sincos} with the better  rate $(1+t)^{\frac{n}{4}-\beta}$ due to the fact that $\frac{n}{4}<1+\frac{1}{2}\lfloor\frac{n}{2}\rfloor$ for any $n\geqslant 1$. We conjecture that the assumption for $\beta\in\{0\}\cup(\frac{1}{2},+\infty)$ is technical from the integer $N$ in the Bernstein theorem.
\end{remark}
\begin{proof}
The first estimate was deduced in \cite[Lemma 3.1 with $\delta=\sigma=1$]{Dao-Reissig=2019} and \cite[Corollary 3]{Narazaki-Reissig=2013}, respectively. In order to improve \cite[Lemma 3.2 with $\delta=\sigma=1$]{Dao-Reissig=2019} under $\beta\in\{0\}\cup(\frac{1}{2},+\infty)$ by our second estimate, we will use a different strategy instead of the modified Bessel function. Note that the boundedness of it for $t\leqslant 1$ is trivial in the second estimate (see, for example, \cite[Lemma 3.2 with $\delta=\sigma=1$ and $t\leqslant 1$]{Dao-Reissig=2019}), and thus we are going to consider $t\geqslant 1$ only. By using the Bernstein theorem (cf. Lemma \ref{Bernstein-Theorem}), we derive
	\begin{align*}
		&\left\| \ml{F}^{-1}_{\xi\to x}\left(\chi_1(|\xi|)\,\mathrm{e}^{-c_1|\xi|^2t}\,|\xi|^{2\beta}g_0(|\xi|t)\right)\right\|_{L^1} \notag \\ 
		& \lesssim \left\| \chi_1(|\xi|)\,\mathrm{e}^{-c_1|\xi|^2t}\,|\xi|^{2\beta}g_0(|\xi|t)\right\|_{L^2}^{1-\frac{n}{2N}}\left(\sum_{|\alpha|=N}\left\| \partial_\xi^{\alpha}\left(\chi_1(|\xi|)\,\mathrm{e}^{-c_1|\xi|^2t}\,|\xi|^{2\beta}g_0(|\xi|t)\right)\right\|_{L^2}\right)^{\frac{n}{2N}}
	\end{align*}
	for $\frac{n}{2}<N\in\mb{N}$ as well as $n\geqslant 1$. From the boundedness of $g_0(|\xi|t)$ and the polar coordinates, one may estimate
	\begin{align*}
		\left\| \chi_1(|\xi|)\,\mathrm{e}^{-c_1|\xi|^2t}\,|\xi|^{2\beta}g_0(|\xi|t)\right\|_{L^2}^2\lesssim\int_0^{2\epsilon_0}\mathrm{e}^{-2c_1r^2t}\,r^{4\beta+n-1}\,\mathrm{d}r\lesssim t^{-\frac{n}{2}-2\beta}.
	\end{align*}
For another, a direct computation follows 
	\begin{align*}
		\left|\partial_\xi^{\alpha}\left(\chi_1(|\xi|)\,\mathrm{e}^{-c_1|\xi|^2t}\,|\xi|^{2\beta}g_0(|\xi|t)\right)\right|\lesssim\chi_1(|\xi|)\,\mathrm{e}^{-c_1|\xi|^2t}\,|\xi|^{2\beta}\times\begin{cases}
		t^{|\alpha|}&\mbox{if}\ \beta=0,\\
		t^{|\alpha|}+|\xi|^{-|\alpha|}&\mbox{if}\ \beta>\frac{1}{2}.
		\end{cases}  
	\end{align*}
	It implies 
	\begin{align*}
		\sum_{|\alpha|=N}\left\| \partial_\xi^{\alpha}\left(\chi_1(|\xi|)\,\mathrm{e}^{-c_1|\xi|^2t}\,|\xi|^{2\beta}g_0(|\xi|t)\right)\right\|_{L^2}^2&\lesssim t^{2N-\frac{n}{2}-2\beta}+\int_0^{2\epsilon_0}\mathrm{e}^{-2c_1r^2t}\,r^{4\beta-2N+n-1}\,\mathrm{d}r\\
		&\lesssim t^{2N-\frac{n}{2}-2\beta}+t^{N-\frac{n}{2}-2\beta}\lesssim t^{2N-\frac{n}{2}-2\beta},
	\end{align*}
via the condition $4\beta-2N+n>0$ if $\beta>\frac{1}{2}$ by choosing $N=\frac{n+1}{2}$ for odd $n$ and $N=\frac{n+2}{2}$ for even $n$ (this is the reason for our restriction on $\beta$), namely,
\begin{align*}
\left\| \ml{F}^{-1}_{\xi\to x}\left(\chi_1(|\xi|)\,\mathrm{e}^{-c_1|\xi|^2t}\,|\xi|^{2\beta}g_0(|\xi|t)\right)\right\|_{L^1}\lesssim t^{\frac{n}{4}-\beta}.
\end{align*}
	Combining the last estimates if $\beta\in\{0\}\cup(\frac{1}{2},+\infty)$ and \cite[Lemma 3.2 with $\delta=\sigma=1$]{Dao-Reissig=2019} if $\beta\in(0,\frac{1}{2}]$, we immediately conclude our desired second estimate.
\end{proof}

Moreover, by following the approach in \cite{Narazaki-Reissig=2013}, we may derive the next sharp estimate to deal with the Fourier multipliers related to $\lambda_1$.
\begin{lemma}\label{lem-oscillating-e}
	Let $n\geqslant 1$ and $\beta\geqslant 0$. The following $L^1$ estimate holds:
	\begin{align*}
		\|I_0(t,\cdot)\|_{L^1}:=\left\| \ml{F}^{-1}_{\xi\to x}\left(\chi_1(|\xi|)\,\mathrm{e}^{-c_1t+c_2|\xi|^2t}\,|\xi|^{2\beta}\right)\right\|_{L^1}\lesssim \mathrm{e}^{-\frac{c_1}{2}t},
	\end{align*}
	with $c_1>0$ and $c_2\in \mathbb{R}$.
\end{lemma}
\begin{proof}
	We next divide our proof into three different cases (cf. \cite[Proposition 4]{Narazaki-Reissig=2013} or \cite[Lemma 3.1]{Dao-Reissig=2019}) with respect to the size of $|x|$ and $t$.
	\par\noindent\textbf{Case 1: $|x|\leqslant 1$ and $t\in [0,1]$.}
	Due to the boundedness of $|x|$ and $t$, the desired estimates obviously hold, to be specific,
	\begin{align*}
	\|I_0(t,\cdot)\|_{L^1(|x|\leqslant 1)}\lesssim\int_{|x|\leqslant 1}\left|\int_{|\xi|\leqslant 2\epsilon_0}\mathrm{e}^{ix\cdot\xi}\,\mathrm{e}^{-c_1t+c_2|\xi|^2t}\,|\xi|^{2\beta}\,\mathrm{d}\xi \right|\mathrm{d}x\lesssim 1.
	\end{align*}
\par\noindent\textbf{Case 2: $|x|\geqslant 1$ and $t\in [0,1]$.}
Let us represent our target via
\begin{align*}
	I_0(t,x)=c\int_0^{+\infty}\chi_1(r)\, \mathrm{e}^{-c_1t+c_2r^2t}\, r^{2\beta+n-1}\widetilde{\ml{J}}_{\frac{n}{2}-1}(r|x|) \, \mathrm{d}r,
\end{align*}
in which we used the modified Bessel functions (cf. Lemma \ref{Modified-Bessel-Funct}) thanks to the radially symmetric of the Fourier multiplier with respect to $|\xi|$. By introducing the vector field $\ml{X}f(r):=\frac{\mathrm{d}}{\mathrm{d}r}(\frac{1}{r}f(r))$, we then carry out $k+1$ steps of integration by parts to obtain 
\begin{align}\label{eq:I-expression-after-partial}
	I_0(t,x)=-\frac{c}{|x|^n}\int_0^{+\infty} \partial_r\left(\ml{X}^k\left(\chi_1(r)\,\mathrm{e}^{-c_1t+c_2r^2t}\, r^{2\beta+2k}\right)\right)\sin(r|x|)\, \mathrm{d}r
\end{align}
with the value $\widetilde{\ml{J}}_{\frac{1}{2}}(s)=-\frac{1}{s}\mathrm{d}_s\widetilde{\ml{J}}_{-\frac{1}{2}}(s)=\sqrt{\frac{2}{\pi}}\frac{\sin s}{s}$ for odd spatial dimensions $n=2k+1$ with $k\geqslant1$. A standard calculation by applying the Leibniz formula leads to 
\begin{align*}
	I_0(t,x)&=\sum_{j=0}^{k}\sum_{\ell=0}^{j+1}\frac{c_{j\ell}}{|x|^n}\int_0^{+\infty}\chi_1^{(\ell)}(r)\, \partial_r^{j+1-\ell}(\mathrm{e}^{-c_1t+c_2r^2t})\, r^{2\beta+j}\sin(r|x|)\, \mathrm{d}r\\ &\quad+\sum_{j=0}^{k}\sum_{\ell=0}^{j}\frac{c_{j\ell}}{|x|^n}\int_0^{+\infty} \chi_1^{(\ell+1)}(r)\,\partial_r^{j-\ell}(\mathrm{e}^{-c_1t+c_2r^2t})\, r^{2\beta+j}\sin(r|x|)\, \mathrm{d}r\\ &\quad +\sum_{j=1}^{k}\sum_{\ell=0}^{j}\frac{c_{j\ell}}{|x|^n}\int_0^{+\infty}\chi_1^{(\ell)}(r)\, \partial_r^{j-\ell}(\mathrm{e}^{-c_1t+c_2r^2t})\, r^{2\beta+j-1}\sin(r|x|)\, \mathrm{d}r
\end{align*}
with some universal constants $c_{j\ell}$. 	Using for $\ell\geqslant 1$ that 
\begin{align*}
	\left|\chi_1^{(\ell)}(r)\,\partial_r^l(\mathrm{e}^{-c_1t+c_2r^2t})\right|\lesssim 1 \ \mbox{for all}\  l\geqslant 0 \ \mbox{and}  \ t\in[0,1]
\end{align*}
on the support of derivatives of $\chi_1(r)$ which is away from $r=0$, one more step of integration by parts yields the upper bound $|x|^{-(n+1)}$ for all integrals with $\ell\geqslant 1$. It remains to study for $j=0,\dots,k$ the integrals
\begin{align*}
	\int_0^{+\infty} \chi_1(r)\,\partial_r^{j+1}(\mathrm{e}^{-c_1t+c_2r^2t})\, r^{2\beta+j}\sin(r|x|)\, \mathrm{d}r.
\end{align*}
Due to the facts that 
\begin{align*}
	\left|\int_0^{|x|^{-1}}\chi_1(r)\,\partial_r^j(\mathrm{e}^{-c_1t+c_2r^2t})\, r^{2\beta+j-1}\sin(r|x|)\, \mathrm{d}r\right|\lesssim\int_0^{|x|^{-1}}r\,\mathrm{d}r\lesssim|x|^{-2}
\end{align*}
and, from an additional integration by parts,
\begin{align*}
	&\left|\int_{|x|^{-1}}^{+\infty}\chi_1(r)\,\partial_r^j(\mathrm{e}^{-c_1t+c_2r^2t})\, r^{2\beta+j-1}\sin(r|x|)\, \mathrm{d}r\right|\\ &\lesssim |x|^{-2}+|x|^{-1}\left|\int_{|x|^{-1}}^{+\infty}\chi_1(r)\,\partial_r\left(\partial_r^j(\mathrm{e}^{-c_1t+c_2r^2t})\, r^{2\beta+j-1}\right)\cos(r|x|)\, \mathrm{d}r\right|\\
	&\lesssim |x|^{-1},
\end{align*}
it immediately yields 
\begin{align*}
	\|I_0(t,\cdot)\|_{L^1(|x|\geqslant 1)}\lesssim \|\,|x|^{-(n+1)}\|_{L^1(|x|\geqslant 1)}\lesssim 1
\end{align*}
for any $t\in[0,1]$ to be our desired bounded estimate for odd spatial dimensions. For another, the deduction for even spatial dimensions $n=2k$ with $k\geqslant 1$ is similar to the above with some slight modifications. Particularly, $\widetilde{\ml{J}}_{\frac{1}{2}}(r|x|)$  in \eqref{eq:I-expression-after-partial} is replaced by $\widetilde{\ml{J}}_{0}(r|x|)$.\par
	\noindent\textbf{Case 3: $|x|\geqslant 0$ and $t\in [1,+\infty)$.}
	Via the changes of variables $\xi=t^{-\frac{1}{2}}\eta$ and $y=t^{-\frac{1}{2}}x$ we have 
    \begin{align*}
        \ml{F}^{-1}_{\xi\to x}\left(\chi_1(|\xi|)\,\mathrm{e}^{-c_1t+c_2|\xi|^2t}\,|\xi|^{2\beta}\right)=t^{-\frac{n}{2}-\beta}\,\mathrm{e}^{-\frac{c_1}{2}t}\,\ml{F}^{-1}_{\eta\to y}\left(\chi_1(t^{-\frac{1}{2}}|\eta|)\,\mathrm{e}^{-\frac{c_1}{2}t+c_2|\eta|^2}\,|\eta|^{2\beta}\right).
    \end{align*}
	Following the similar procedure to \cite[Proposition 4, particularly, Formula (58)]{Narazaki-Reissig=2013} by changing their exponential decay term into $\mathrm{e}^{-\frac{c_1}{2}t+c_2|\eta|^2}$ we are able to obtain
	\begin{align*}
		\left\|\ml{F}^{-1}_{\eta\to y}\left(\chi_1(t^{-\frac{1}{2}}|\eta|)\,\mathrm{e}^{-\frac{c_1}{2}t+c_2|\eta|^2}\,|\eta|^{2\beta}\right)\right\|_{L^1}\lesssim 1.
	\end{align*}
	Summing up all derived estimates in the above we complete the proof.
\end{proof}

	\subsection{$L^p-L^q$ estimates for small frequencies}\label{Sub-Section-Small}
	\hspace{5mm}To investigate $L^p-L^q$ estimates for the solutions localizing in the small frequency zone, we next provide some $L^1$ and $L^\infty$ estimates, respectively, for the kernels.
	\begin{prop}\label{prop:L^1-small}
		Let $n\geqslant 1$ and $s\geqslant 0$. Then, the kernels satisfy the following $L^1$ estimates:
		\begin{align*}
			\left\| \ml{F}^{-1}_{\xi\to x}\left(\chi_1(|\xi|)|\xi|^s\widehat{K}_0(t,|\xi|)\right)\right\|_{L^1}&\lesssim\begin{cases}
			(1+t)^{\frac{1}{2}(1+\lfloor\frac{n}{2}\rfloor)}&\mbox{if}\ s=0,\\
			(1+t)^{\frac{1}{2}(2+\lfloor\frac{n}{2}\rfloor)-\frac{s}{2}}&\mbox{if}\ s\in(0,1],\\
					(1+t)^{\frac{n}{4}-\frac{s}{2}}&\mbox{if}\ s\in(1,+\infty),
			\end{cases}\\
			\left\| \ml{F}^{-1}_{\xi\to x}\left(\chi_1(|\xi|)|\xi|^s\widehat{K}_1(t,|\xi|)\right)\right\|_{L^1}&\lesssim\begin{cases}
				(1+t)^{\frac{1}{2}(3+\lfloor\frac{n}{2}\rfloor)-\frac{s}{2}}&\mbox{if}\ s\in[0,1)\cup(1,2],\\
				(1+t)^{\frac{n}{4}+\frac{1}{2}-\frac{s}{2}}&\mbox{if}\ s\in\{1\}\cup(2,+\infty),
			\end{cases}\\
					\left\| \ml{F}^{-1}_{\xi\to x}\left(\chi_1(|\xi|)|\xi|^s\widehat{K}_2(t,|\xi|)\right)\right\|_{L^1}&\lesssim\begin{cases}
			(1+t)^{\frac{1}{2}(3+\lfloor\frac{n}{2}\rfloor)-\frac{s}{2}}&\mbox{if}\ s\in[0,1)\cup(1,2],\\
			(1+t)^{\frac{1}{2}(1+\lfloor\frac{n}{2}\rfloor)}&\mbox{if}\ s=1,\\			
			(1+t)^{\frac{n}{4}+\frac{1}{2}-\frac{s}{2}}&\mbox{if}\ s\in(2,+\infty).
		\end{cases}
		\end{align*}
 Furthermore, by subtracting the corresponding profiles of these kernels, the error terms satisfy the following refined $L^1$ estimates:
		\begin{align*}
			\left\| \ml{F}^{-1}_{\xi \to x}\left(\chi_1(|\xi|)|\xi|^{s}\left(\widehat{K}_1(t,|\xi|)-\widehat{J}(t,|\xi|)\right)\right)\right\|_{L^1}&\lesssim \begin{cases}
				(1+t)^{\frac{1}{2}(1+\lfloor\frac{n}{2}\rfloor)}&\mbox{if}\ s=0,\\
				(1+t)^{\frac{n}{4}-\frac{1}{2}-\frac{s}{2}}&\mbox{if}\ s\in(0,+\infty),
			\end{cases}\\
			\left\| \ml{F}^{-1}_{\xi \to x}\left(\chi_1(|\xi|)|\xi|^{s}\left(\widehat{K}_2(t,|\xi|)-\tau\widehat{J}(t,|\xi|)\right)\right)\right\|_{L^1}&\lesssim\begin{cases}
				(1+t)^{\frac{1}{2}(1+\lfloor\frac{n}{2}\rfloor)}&\mbox{if}\ s=0,\\
(1+t)^{\frac{1}{2}(2+\lfloor\frac{n}{2}\rfloor)-\frac{s}{2}}&\mbox{if}\ s\in(0,1],\\
	(1+t)^{\frac{n}{4}-\frac{s}{2}}&\mbox{if}\ s\in(1,+\infty),
			\end{cases}
		\end{align*}
		where the singular diffusion waves kernel is defined via
		\begin{align*}
		\widehat{J}(t,|\xi|):=\frac{\sin(|\xi|t)}{|\xi|}\,\mathrm{e}^{-\frac{\delta}{2}|\xi|^2t}.
		\end{align*}
	\end{prop}
\begin{remark}\label{Rem-01}
For the sake of convenient, a direct computation implies that the totally estimates are determined by
\begin{align*}
\sum\limits_{\ell=0,1,2}\left\| \ml{F}^{-1}_{\xi\to x}\left(\chi_1(|\xi|)|\xi|^s\widehat{K}_{\ell}(t,|\xi|)\right)\right\|_{L^1}\lesssim\begin{cases}
	(1+t)^{\frac{1}{2}(3+\lfloor\frac{n}{2}\rfloor)-\frac{s}{2}}&\mbox{if}\ s\in[0,1)\cup(1,2],\\
	(1+t)^{\frac{1}{2}(1+\lfloor\frac{n}{2}\rfloor)}&\mbox{if}\ s=1,\\			
	(1+t)^{\frac{n}{4}+\frac{1}{2}-\frac{s}{2}}&\mbox{if}\ s\in(2,+\infty),
\end{cases}
\end{align*}
moreover,
\begin{align*}
&\sum\limits_{\ell=1,2}\left\| \ml{F}^{-1}_{\xi\to x}\left(\chi_1(|\xi|)|\xi|^s\big(\widehat{K}_{\ell}(t,|\xi|)-[2-\ell+(\ell-1)\tau]\widehat{J}(t,|\xi|)\big)\right)\right\|_{L^1}\\
&\lesssim\begin{cases}
	(1+t)^{\frac{1}{2}(1+\lfloor\frac{n}{2}\rfloor)}&\mbox{if}\ s=0,\\
	(1+t)^{\frac{1}{2}(2+\lfloor\frac{n}{2}\rfloor)-\frac{s}{2}}&\mbox{if}\ s\in(0,1],\\
	(1+t)^{\frac{n}{4}-\frac{s}{2}}&\mbox{if}\ s\in(1,+\infty).
\end{cases}
\end{align*}
\end{remark}
\begin{proof}
	With the aim of deriving our desired estimates, we will apply Lemmas \ref{lem-oscillating-sincos} and \ref{lem-oscillating-e} for all elements in asymptotic expansions of kernels in \eqref{widehat-K-expression}. We first rewrite the kernel as
	\begin{align*}
		\widehat{K}_0^1&=\frac{-(|\xi|^2+\tau\delta|\xi|^4+O(|\xi|^6))\, \mathrm{e}^{-\frac{1}{\tau}t+\delta|\xi|^2t+O(|\xi|^4)t}}{-\frac{1}{\tau^2}-|\xi|^2+\frac{3\delta}{\tau}|\xi|^2+O(|\xi|^4)}\\
		&=\tau^2|\xi|^2\,\mathrm{e}^{-\frac{1}{\tau}t+\delta|\xi|^2t}+\tau^3(4\delta-\tau)|\xi|^4\,\mathrm{e}^{-\frac{1}{\tau}t+\delta|\xi|^2t}+O(|\xi|^6)\, \mathrm{e}^{-\frac{1}{\tau}t+\delta|\xi|^2t}.
	\end{align*}
Then, by applying Lemma \ref{lem-oscillating-e} with suitable $\beta=\beta(s)$ and $c_1=\frac{1}{\tau}$, $c_2=\delta$, one obtains 
\begin{align*}
	\left\| \ml{F}^{-1}_{\xi \to x}\left(\chi_1(|\xi|)|\xi|^s\widehat{K}_0^1(t,|\xi|)\right)\right\|_{L^1}\lesssim \mathrm{e}^{-\frac{1}{2\tau}t}.
\end{align*}
According to asymptotic expansions for small frequencies, analogously,
\begin{align}\label{eq:K_1^1-expression}
	\widehat{K}_1^1&=\tau^2\delta|\xi|^2\,\mathrm{e}^{-\frac{1}{\tau}t+\delta|\xi|^2t}+2\tau^3\delta(2\delta-\tau)|\xi|^4\,\mathrm{e}^{-\frac{1}{\tau}t+\delta|\xi|^2t}+O(|\xi|^6)\, \mathrm{e}^{-\frac{1}{\tau}t+\delta|\xi|^2t},\\ \notag \widehat{K}_2^1&=\tau^2\,\mathrm{e}^{-\frac{1}{\tau}t+\delta|\xi|^2t}+\tau^3(3\delta-\tau)|\xi|^2\,\mathrm{e}^{-\frac{1}{\tau}t+\delta|\xi|^2t}+O(|\xi|^4)\, \mathrm{e}^{-\frac{1}{\tau}t+\delta|\xi|^2t},
\end{align}
we may estimate for $\ell=1,2$ that
\begin{align*}
	\left\| \ml{F}^{-1}_{\xi \to x}\left(\chi_1(|\xi|)|\xi|^s\widehat{K}_\ell^1(t,|\xi|)\right)\right\|_{L^1}\lesssim \mathrm{e}^{-\frac{1}{2\tau}t}.
\end{align*}
Considering the expansion
\begin{align*}
	\widehat{K}_0^{\cos}&= \left(1-\tau^2|\xi|^2+O(|\xi|^4)\right)\mathrm{e}^{-\frac{\delta}{2}|\xi|^2t}\, \sum_{k=0}^{+\infty}\frac{(O(|\xi|^4)t)^k}{k!}\, \sum_{l=0}^{+\infty}\frac{\cos^{(l)}(|\xi|t)}{l!}\left(\frac{\delta(4\tau-\delta)}{8}|\xi|^3t\right)^l\\ 
	&=\cos(|\xi|t)\, \mathrm{e}^{-\frac{\delta}{2}|\xi|^2t}-\tau^2|\xi|^2\cos(|\xi|t)\, \mathrm{e}^{-\frac{\delta}{2}|\xi|^2t}-\frac{\delta(4\tau-\delta)}{8}\sin(|\xi|t)|\xi|^3t\, \mathrm{e}^{-\frac{\delta}{2}|\xi|^2t}\\ &\quad-\frac{\tau\delta(\delta-\tau)}{2}|\xi|^4t\cos(|\xi|t)\, \mathrm{e}^{-\frac{\delta}{2}|\xi|^2t}+O(|\xi|^5)t\big(|\xi|\cos(|\xi|t)+\sin(|\xi|t)\big)\,\mathrm{e}^{-\frac{\delta}{2}|\xi|^2t},
\end{align*}
associated with the second estimate in Lemma \ref{lem-oscillating-sincos}, we conclude 
\begin{align*}
	\left\| \ml{F}^{-1}_{\xi \to x}\left(\chi_1(|\xi|)|\xi|^s\widehat{K}_0^{\cos}(t,|\xi|)\right)\right\|_{L^1}\lesssim\begin{cases}
		(1+t)^{\frac{n}{4}-\frac{s}{2}}&\mbox{if}\ s\in\{0\}\cup(1,+\infty),\\
		(1+t)^{\frac{1}{2}(2+\lfloor\frac{n}{2}\rfloor)-\frac{s}{2}}&\mbox{if}\ s\in(0,1].
	\end{cases}
\end{align*}
By the same way, the next asymptotic expansions hold:
\begin{align}\label{eq:K_1^cos-expression}
	\notag \widehat{K}_1^{\cos}&=-\tau^2\delta|\xi|^2\cos(|\xi|t)\, \mathrm{e}^{-\frac{\delta}{2}|\xi|^2t}-2\tau^3\delta(2\delta-\tau)|\xi|^4\cos(|\xi|t)\, \mathrm{e}^{-\frac{\delta}{2}|\xi|^2t}+\frac{\tau^3\delta^2(\delta-\tau)}{2}|\xi|^6t\cos(|\xi|t)\, \mathrm{e}^{-\frac{\delta}{2}|\xi|^2t}\\ 
	&\quad +\frac{\tau^2\delta^2(4\tau-\delta)}{8}\sin(|\xi|t)|\xi|^5t\, \mathrm{e}^{-\frac{\delta}{2}|\xi|^2t}+O(|\xi|^7)t\big(|\xi|\cos(|\xi|t)+\sin(|\xi|t)\big)\,\mathrm{e}^{-\frac{\delta}{2}|\xi|^2t},\\ \notag
	\widehat{K}_2^{\cos}&=-\tau^2\cos(|\xi|t)\, \mathrm{e}^{-\frac{\delta}{2}|\xi|^2t}-\tau^3(3\delta-\tau)|\xi|^2\cos(|\xi|t)\, \mathrm{e}^{-\frac{\delta}{2}|\xi|^2t}+\frac{\tau^3\delta(\delta-\tau)}{2}|\xi|^4t\cos(|\xi|t)\, \mathrm{e}^{-\frac{\delta}{2}|\xi|^2t}\\ &\notag\quad+\frac{\tau^2\delta(4\tau-\delta)}{8}\sin(|\xi|t)|\xi|^3t\, \mathrm{e}^{-\frac{\delta}{2}|\xi|^2t}+O(|\xi|^5)t\big(|\xi|\cos(|\xi|t)+\sin(|\xi|t)\big)\,\mathrm{e}^{-\frac{\delta}{2}|\xi|^2t},
\end{align}
which imply
\begin{align*}
	\left\| \ml{F}^{-1}_{\xi \to x}\left(\chi_1(|\xi|)|\xi|^s\widehat{K}_1^{\cos}(t,|\xi|)\right)\right\|_{L^1}&\lesssim (1+t)^{\frac{n}{4}-1-\frac{s}{2}},\\
		\left\| \ml{F}^{-1}_{\xi \to x}\left(\chi_1(|\xi|)|\xi|^s\widehat{K}_2^{\cos}(t,|\xi|)\right)\right\|_{L^1}&\lesssim\begin{cases}
		(1+t)^{\frac{n}{4}-\frac{s}{2}}&\mbox{if}\ s\in\{0\}\cup(1,+\infty),\\
		(1+t)^{\frac{1}{2}(2+\lfloor\frac{n}{2}\rfloor)-\frac{s}{2}}&\mbox{if}\ s\in(0,1].
	\end{cases}
\end{align*}
Furthermore, with the aid of
\begin{align*}
	\widehat{K}_0^{\sin}&=\left(\frac{\delta+2\tau}{2}|\xi|+\frac{\delta^3-16\tau^3+6\tau\delta^2+24\tau^2\delta}{16}|\xi|^3+O(|\xi|^5)\right) \mathrm{e}^{-\frac{\delta}{2}|\xi|^2t}\, \sum_{k= 0}^{+\infty}\frac{(O(|\xi|^4)t)^k}{k!}\\
	&\quad\times\sum_{l=0}^{+\infty}\frac{\sin^{(l)}(|\xi|t)}{l!}\left(\frac{\delta(4\tau-\delta)}{8}|\xi|^3t\right)^l\\
	&=\frac{\delta+2\tau}{2}|\xi|\sin(|\xi|t)\, \mathrm{e}^{-\frac{\delta}{2}|\xi|^2t}+O(|\xi|^3)t\big(|\xi|\cos(|\xi|t)+\sin(|\xi|t)\big)\,\mathrm{e}^{-\frac{\delta}{2}|\xi|^2t},
\end{align*}
it yields
\begin{align*}
	\left\| \ml{F}^{-1}_{\xi \to x}\left(\chi_1(|\xi|)|\xi|^s\widehat{K}_0^{\sin}(t,|\xi|)\right)\right\|_{L^1}\lesssim
	\begin{cases}
	(1+t)^{\frac{1}{2}(1+\lfloor\frac{n}{2}\rfloor)}&\mbox{if}\ s=0,\\
	(1+t)^{\frac{n}{4}-\frac{1}{2}-\frac{s}{2}}&\mbox{if}\ s\in(0,+\infty).
	\end{cases}
\end{align*}
For the other terms, due to the expansions that 
\begin{align}
	\notag \widehat{K}_1^{\sin}&=\frac{\sin(|\xi|t)}{|\xi|}\,\mathrm{e}^{-\frac{\delta}{2}|\xi|^2t}+\frac{\delta^2+4\tau\delta}{8}|\xi|\sin(|\xi|t)\, \mathrm{e}^{-\frac{\delta}{2}|\xi|^2t}-\frac{\tau\delta(\delta-\tau)}{2}|\xi|^3t\sin(|\xi|t)\, \mathrm{e}^{-\frac{\delta}{2}|\xi|^2t}\\ \label{eq:K_1^sin-expression}
	&\quad+\frac{\delta(4\tau-\delta)}{8}|\xi|^2t\cos(|\xi|t)\, \mathrm{e}^{-\frac{\delta}{2}|\xi|^2t}+O(|\xi|^4)t\big(\cos(|\xi|t)+|\xi|\sin(|\xi|t)\big)\,\mathrm{e}^{-\frac{\delta}{2}|\xi|^2t},\\ \notag
	\widehat{K}_2^{\sin}&=\tau\frac{\sin(|\xi|t)}{|\xi|}\,\mathrm{e}^{-\frac{\delta}{2}|\xi|^2t}+\frac{\delta^2+8\tau\delta-8\tau^2}{8}|\xi|\sin(|\xi|t)\, \mathrm{e}^{-\frac{\delta}{2}|\xi|^2t}-\frac{\tau^2\delta(\delta-\tau)}{2}|\xi|^3t\sin(|\xi|t)\, \mathrm{e}^{-\frac{\delta}{2}|\xi|^2t}\\ &\quad+\frac{\tau\delta(4\tau-\delta)}{8}|\xi|^2t\cos(|\xi|t)\, \mathrm{e}^{-\frac{\delta}{2}|\xi|^2t}+O(|\xi|^4)t\big(\cos(|\xi|t)+|\xi|\sin(|\xi|t)\big)\,\mathrm{e}^{-\frac{\delta}{2}|\xi|^2t},\notag
\end{align}
from Lemma \ref{lem-oscillating-sincos} with suitable $\beta=\beta(s)$, one directly claims
\begin{align*}
	\left\| \ml{F}^{-1}_{\xi \to x}\left(\chi_1(|\xi|)|\xi|^s\widehat{K}_\ell^{\sin}(t,|\xi|)\right)\right\|_{L^1}\lesssim
	\begin{cases}
		(1+t)^{\frac{1}{2}(2+\lfloor\frac{n}{2}\rfloor)+\frac{1}{2}-\frac{s}{2}} &\mbox{if} \  s\in[0,1)\cup  (1,2],\\
		(1+t)^{\frac{n}{4}+\frac{1}{2}-\frac{s}{2}}&\mbox{if} \ s\in\{1\}\cup(2,+\infty),
	\end{cases}
\end{align*}
for $\ell=1,2$. Furthermore, by subtracting the corresponding leading terms of sine kernels and using Lemma \ref{lem-oscillating-sincos} we obtain for $\ell=1,2$ that
\begin{align*}
	\left\| \ml{F}^{-1}_{\xi\to x}\left(\chi_1(|\xi|)|\xi|^s\big(\widehat{K}_\ell^{\sin}(t,|\xi|)-[2-\ell+(\ell-1)\tau]\widehat{J}(t,|\xi|)\big)\right)\right\|_{L^1}\lesssim \begin{cases}
	(1+t)^{\frac{1}{2}(1+\lfloor\frac{n}{2}\rfloor)}&\mbox{if}\ s=0,\\
	(1+t)^{\frac{n}{4}-\frac{1}{2}-\frac{s}{2}}&\mbox{if}\ s\in(0,+\infty).
	\end{cases}
\end{align*}
	    Finally, thanks to \eqref{widehat-K-expression} associated with all derived estimates in the above, carrying out some comparisons, the proofs are immediately completed.
\end{proof}
	
\begin{prop}\label{prop:L^infty-small}
		Let $n\geqslant 2$ and $s\geqslant 0$. Then, the kernels satisfy the following $L^{\infty}$ estimates:
		\begin{align*}
			\left\| \ml{F}^{-1}_{\xi \to x}\left(\chi_1(|\xi|)|\xi|^{s}\widehat{K}_\ell(t,|\xi|)\right)\right\|_{L^\infty}\lesssim\begin{cases}
				(1+t)^{-\frac{n}{2}-\frac{s}{2}}&\mbox{if}\ \ell=0,\\
				(1+t)^{-\frac{n}{2}+\frac{1}{2}-\frac{s}{2}}&\mbox{if}\ \ell=1,\\
				(1+t)^{-\frac{n}{2}+\frac{1}{2}-\frac{s}{2}}&\mbox{if}\ \ell=2.
			\end{cases}
		\end{align*}
 Furthermore, by subtracting the corresponding profiles of these kernels, the error terms satisfy the following refined $L^{\infty}$ estimates:
		\begin{align*}
			\left\| \ml{F}^{-1}_{\xi \to x}\left(\chi_1(|\xi|)|\xi|^{s}\left(\widehat{K}_\ell(t,|\xi|)-[2-\ell+(\ell-1)\tau]\widehat{J}(t,|\xi|)\right)\right)\right\|_{L^\infty}\lesssim (1+t)^{-\frac{n}{2}-\frac{s}{2}}
		\end{align*}
	with $\ell=1,2$.
	\end{prop}
\begin{remark}\label{Rem-02}
	For the sake of convenient, a direct computation implies that the totally estimates are determined by
	\begin{align*}
		\sum\limits_{\ell=0,1,2}\left\| \ml{F}^{-1}_{\xi\to x}\left(\chi_1(|\xi|)|\xi|^s\widehat{K}_{\ell}(t,|\xi|)\right)\right\|_{L^{\infty}}\lesssim (1+t)^{-\frac{n}{2}+\frac{1}{2}-\frac{s}{2}},
	\end{align*}
	moreover,
	\begin{align*}
		\sum\limits_{\ell=1,2}\left\| \ml{F}^{-1}_{\xi\to x}\left(\chi_1(|\xi|)|\xi|^s\big(\widehat{K}_{\ell}(t,|\xi|)-[2-\ell+(\ell-1)\tau]\widehat{J}(t,|\xi|)\big)\right)\right\|_{L^{\infty}}\lesssim(1+t)^{-\frac{n}{2}-\frac{s}{2}}.
	\end{align*}
\end{remark}
	\begin{proof}
		Thanks to asymptotic behavior for the kernels in the Fourier space, i.e. \eqref{eq:K_1^1-expression}, \eqref{eq:K_1^cos-expression} and \eqref{eq:K_1^sin-expression}, we can obtain the pointwise estimates
		\begin{align*}
			\chi_1(|\xi|)|\widehat{K}_1(t,|\xi|)| &\leqslant\chi_1(|\xi|)\left( |\widehat{K}_1^1(t,|\xi|)|+|\widehat{K}_1^{\cos}(t,|\xi|)|+|\widehat{K}_1^{\sin}(t,|\xi|)|\right)\\
			&\lesssim \chi_1(|\xi|)\left(|\xi|^2\,\mathrm{e}^{-ct}+|\xi|^2|\cos(|\xi|t)|\,\mathrm{e}^{-c|\xi|^2t}+|\xi|^{-1}|\sin(|\xi|t)|\,\mathrm{e}^{-c|\xi|^2t}\right)\\
			&\lesssim\chi_1(|\xi|) |\xi|^{-1}\,\mathrm{e}^{-c|\xi|^2t}.
		\end{align*}
		It immediately gives
		\begin{align*}
			\left\|\ml{F}^{-1}_{\xi\to x}\left(\chi_1(|\xi|)|\xi|^s\widehat{K}_1(t,|\xi|)\right)\right\|_{L^\infty}&\lesssim\left\|\chi_1(|\xi|) |\xi|^s\widehat{K}_1(t,|\xi|)\right\|_{L^1}\\
			& \lesssim \int_0^{2\epsilon_0}\mathrm{e}^{-cr^2t}\,r^{n+s-2}\,\mathrm{d}r\lesssim (1+t)^{-\frac{n}{2}+\frac{1}{2}-\frac{s}{2}}
		\end{align*}
		for $n\geqslant 2$ due to $n+s-2>-1$. Additionally, the subtraction with its leading term shows
		\begin{align*}
			\chi_1(|\xi|)|\widehat{K}_1(t,|\xi|)-\widehat{J}(t,|\xi|)| \lesssim \chi_1(|\xi|)\left(\mathrm{e}^{-ct}+\mathrm{e}^{-c|\xi|^2t}\right),
		\end{align*}
		which leads to 
		\begin{align*}
			\left\|\ml{F}^{-1}_{\xi\to x}\left(\chi_1(|\xi|)|\xi|^s\big(\widehat{K}_1(t,|\xi|)-\widehat{J}(t,|\xi|)\big)\right)\right\|_{L^\infty}\lesssim (1+t)^{-\frac{n}{2}-\frac{1}{2}-\frac{s}{2}}.
		\end{align*}
		The other cases can be followed by the last computations associated with
		\begin{align*}
		\chi_{1}(|\xi|)|\widehat{K}_0(t,|\xi|)|&\lesssim \chi_{1}(|\xi|)\left(\mathrm{e}^{-ct}+\mathrm{e}^{-c|\xi|^2t}\right),\\
		\chi_{1}(|\xi|)|\widehat{K}_2(t,|\xi|)|&\lesssim\chi_{1}(|\xi|)\left(\mathrm{e}^{-ct}+|\xi|^{-1}\,\mathrm{e}^{-c|\xi|^2t}\right),\\
		\chi_1(|\xi|)|\widehat{K}_2(t,|\xi|)-\tau\widehat{J}(t,|\xi|)| &\lesssim \chi_1(|\xi|)\left(\mathrm{e}^{-ct}+\mathrm{e}^{-c|\xi|^2t}\right).
		\end{align*}
	The proof is complete.
	\end{proof}
\begin{remark}
	The estimates in \cite[Proposition 3.4]{Dao-Reissig=2019} can be improved via $(1+t)^{-\frac{1}{2}}$ if $n\geqslant 2$ by following our approach where we used the boundedness of $\sin(|\xi|t)$ instead of $\sin(|\xi|t)\leqslant |\xi|t$.
\end{remark}

From Propositions \ref{prop:L^1-small} and \ref{prop:L^infty-small} (see also Remarks \ref{Rem-01} and \ref{Rem-02}), by employing the Riesz-Thorin interpolation argument we may conclude the following $L^r$ estimates.
\begin{prop}\label{prop:L^r-small}
	Let $n\geqslant 2$, $s\geqslant0$ and $r\in[1,+\infty]$. Then, the kernels satisfy the following $L^r$ estimates:
	\begin{align*}
		\sum\limits_{\ell=0,1,2}\left\| \ml{F}^{-1}_{\xi\to x}\left(\chi_1(|\xi|)|\xi|^s\widehat{K}_{\ell}(t,|\xi|)\right)\right\|_{L^r}\lesssim\begin{cases}
			(1+t)^{(1+\frac{n}{2}+\frac{1}{2}\lfloor\frac{n}{2}\rfloor)\frac{1}{r}-\frac{n}{2}+\frac{1}{2}-\frac{s}{2}}&\mbox{if}\ s\in[0,1)\cup(1,2],\\
			(1+t)^{(\frac{1}{2}+\frac{n}{2}+\frac{1}{2}\lfloor\frac{n}{2}\rfloor)\frac{1}{r}-\frac{n}{2}}&\mbox{if}\ s=1,\\			
			(1+t)^{\frac{3n}{4r}-\frac{n}{2}+\frac{1}{2}-\frac{s}{2}}&\mbox{if}\ s\in(2,+\infty).
		\end{cases}
	\end{align*}
Furthermore, by subtracting the corresponding profiles of these kernels, the error terms satisfy the following refined $L^r$ estimates:
\begin{align*}
	&\sum\limits_{\ell=1,2}\left\| \ml{F}^{-1}_{\xi\to x}\left(\chi_1(|\xi|)|\xi|^s\big(\widehat{K}_{\ell}(t,|\xi|)-[2-\ell+(\ell-1)\tau]\widehat{J}(t,|\xi|)\big)\right)\right\|_{L^r}\\
	&\lesssim\begin{cases}
		(1+t)^{(\frac{1}{2}+\frac{n}{2}+\frac{1}{2}\lfloor\frac{n}{2}\rfloor)\frac{1}{r}-\frac{n}{2}}&\mbox{if}\ s=0,\\
		(1+t)^{(1+\frac{n}{2}+\frac{1}{2}\lfloor\frac{n}{2}\rfloor)\frac{1}{r}-\frac{n}{2}-\frac{s}{2}}&\mbox{if}\ s\in(0,1],\\
		(1+t)^{\frac{3n}{4r}-\frac{n}{2}-\frac{s}{2}}&\mbox{if}\ s\in(1,+\infty).
	\end{cases}
\end{align*}
\end{prop}

An application of the Young convolution inequality to
\begin{align*}
|D|^s\varphi^{>0}(t,x)=\sum\limits_{\ell=0,1,2}|D|^sK_{\ell}(t,x)\ast_{(x)}\varphi_{\ell}^{>0}(x)
\end{align*}
in the $L^r$ norm with $\frac{1}{r}=1+\frac{1}{q}-\frac{1}{p}$ and Proposition \ref{prop:L^r-small} concludes the next result.
\begin{prop}\label{prop:L^m-L^q-small}
Let $1\leqslant p\leqslant q\leqslant +\infty$ and $s\geqslant0$.	The solution localizing in the small frequency zone to the dissipative MGT equation \eqref{Eq_MGT} with $\delta>0$ for $n\geqslant 2$ satisfies the following $L^p-L^q$ estimates:
	\begin{align*}
		\|\varphi^{>0}(t,\cdot)\|_{\dot{H}^{s}_{q,\chi_1}}\lesssim \begin{cases}
			(1+t)^{-(1+\frac{n}{2}+\frac{1}{2}\lfloor\frac{n}{2}\rfloor)(\frac{1}{p}-\frac{1}{q})+\frac{3}{2}+\frac{1}{2}\lfloor\frac{n}{2}\rfloor-\frac{s}{2}}\|(\varphi_0^{>0},\varphi_1^{>0},\varphi_2^{>0})\|_{(L^p)^3}&\mbox{if}\ s\in[0,1)\cup(1,2],\\
			(1+t)^{-(\frac{1}{2}+\frac{n}{2}+\frac{1}{2}\lfloor\frac{n}{2}\rfloor)(\frac{1}{p}-\frac{1}{q})+\frac{1}{2}+\frac{1}{2}\lfloor\frac{n}{2}\rfloor}\|(\varphi_0^{>0},\varphi_1^{>0},\varphi_2^{>0})\|_{(L^p)^3}&\mbox{if}\ s=1,\\			
			(1+t)^{-\frac{3n}{4}(\frac{1}{p}-\frac{1}{q})+\frac{1}{2}+\frac{n}{4}-\frac{s}{2}}\|(\varphi_0^{>0},\varphi_1^{>0},\varphi_2^{>0})\|_{(L^p)^3}&\mbox{if}\ s\in(2,+\infty).
		\end{cases}
	\end{align*}
	Furthermore, by subtracting the profile
	\begin{align*}
	\Psi(t,x):=\ml{F}^{-1}_{\xi\to x}\left(\frac{\sin(|\xi|t)}{|\xi|}\,\mathrm{e}^{-\frac{\delta}{2}|\xi|^2t}\right)\left(\varphi_1^{>0}(x)+\tau\varphi_2^{>0}(x)\right),
	\end{align*} the following refined estimates hold:
	\begin{align*}
		&\| (\varphi^{>0}-\Psi)(t,\cdot)\|_{\dot{H}^{s}_{q,\chi_1}}\lesssim\begin{cases}
			(1+t)^{-(\frac{1}{2}+\frac{n}{2}+\frac{1}{2}\lfloor\frac{n}{2}\rfloor)(\frac{1}{p}-\frac{1}{q})+\frac{1}{2}+\frac{1}{2}\lfloor\frac{n}{2}\rfloor}\|(\varphi_0^{>0},\varphi_1^{>0},\varphi_2^{>0})\|_{(L^p)^3}&\mbox{if}\ s=0,\\
			(1+t)^{-(1+\frac{n}{2}+\frac{1}{2}\lfloor\frac{n}{2}\rfloor)(\frac{1}{p}-\frac{1}{q})+1+\frac{1}{2}\lfloor\frac{n}{2}\rfloor-\frac{s}{2}}\|(\varphi_0^{>0},\varphi_1^{>0},\varphi_2^{>0})\|_{(L^p)^3}&\mbox{if}\ s\in(0,1],\\
			(1+t)^{-\frac{3n}{4}(\frac{1}{p}-\frac{1}{q})+\frac{n}{4}-\frac{s}{2}}\|(\varphi_0^{>0},\varphi_1^{>0},\varphi_2^{>0})\|_{(L^p)^3}&\mbox{if}\ s\in(1,+\infty).
		\end{cases}
	\end{align*}
\end{prop}

\subsection{$L^q-L^q$ estimates for large frequencies}\label{Sub-Section-Large}
\hspace{5mm} As our preparations for applications of the Mikhlin-H\"omander multiplier theorem (cf. Lemma \ref*{Minklin-Homander-Theorem}) in the forthcoming part, let us state some sharp estimates for the multipliers related to the kernels in the Fourier space.
\begin{lemma}\label{lem:auxiliary-lem-large}
	Let $\alpha\in \mb{N}_0^n$ and $\sigma\geqslant 0$. The following pointwise estimates hold in the large frequency zone $\{\xi\in\mb{R}^n:|\xi|\geqslant N_0\}$:
	\begin{align}
		\label{lem-auxiliary-profile}&\left|\partial_{\xi}^\alpha\left(|\xi|^{\sigma}\sin(|\xi|t )\, \mathrm{e}^{-\frac{\delta}{2}|\xi|^2t}\right)\right|\lesssim \mathrm{e}^{-c|\xi|^2t}\,|\xi|^{\sigma-|\alpha|},\\ \label{lem-auxiliary-2-1}&\left|\partial_{\xi}^\alpha\frac{-\mathrm{e}^{\lambda_1t}\,|\xi|^\sigma}{2\mu_{\mathrm{R}}\lambda_1-\mu_{\mathrm{I}}^2-\mu_{\mathrm{R}}^2-\lambda_1^2}\right|\lesssim \mathrm{e}^{-ct}\,|\xi|^{-2+\sigma-|\alpha|},\\
		\label{lem-auxiliary-2-cos}&\left|\partial_{\xi}^\alpha\frac{\cos(\mu_{\mathrm{I}}t )\, \mathrm{e}^{\mu_{\mathrm{R}} t}\,|\xi|^\sigma}{2\mu_{\mathrm{R}}\lambda_1-\mu_{\mathrm{I}}^2-\mu_{\mathrm{R}}^2-\lambda_1^2}\right|\lesssim \mathrm{e}^{-ct}\,|\xi|^{-2+\sigma},\\
		\label{lem-auxiliary-2-sin}&\left|\partial_{\xi}^\alpha\frac{-(\mu_{\mathrm{R}}-\lambda_1)\sin(\mu_{\mathrm{I}}t )\, \mathrm{e}^{\mu_{\mathrm{R}} t}\,|\xi|^\sigma}{\mu_{\mathrm{I}}(2\mu_{\mathrm{R}}\lambda_1-\mu_{\mathrm{I}}^2-\mu_{\mathrm{R}}^2-\lambda_1^2)}\right|\lesssim \mathrm{e}^{-ct}\,|\xi|^{-3+\sigma},\\
		\notag&\left|\partial_{\xi}^\alpha\frac{2\mu_{\mathrm{R}}\, \mathrm{e}^{\lambda_1t}\,|\xi|^\sigma}{2\mu_{\mathrm{R}}\lambda_1-\mu_{\mathrm{I}}^2-\mu_{\mathrm{R}}^2-\lambda_1^2}\right|\lesssim \mathrm{e}^{-ct}\,|\xi|^{-2+\sigma-|\alpha|},\\
		\notag&\left|\partial_{\xi}^\alpha\frac{-2\mu_{\mathrm{R}}\cos(\mu_{\mathrm{I}}t )\, \mathrm{e}^{\mu_{\mathrm{R}} t}\,|\xi|^\sigma}{2\mu_{\mathrm{R}}\lambda_1-\mu_{\mathrm{I}}^2-\mu_{\mathrm{R}}^2-\lambda_1^2}\right|\lesssim \mathrm{e}^{-ct}\,|\xi|^{-2+\sigma},\\
		\notag&\left|\partial_{\xi}^\alpha\frac{(\mu_{\mathrm{R}}^2-\mu_{\mathrm{I}}^2-\lambda_1^2)\sin(\mu_{\mathrm{I}}t )\, \mathrm{e}^{\mu_{\mathrm{R}} t}\,|\xi|^\sigma}{\mu_{\mathrm{I}}(2\mu_{\mathrm{R}}\lambda_1-\mu_{\mathrm{I}}^2-\mu_{\mathrm{R}}^2-\lambda_1^2)}\right|\lesssim \mathrm{e}^{-ct}\,|\xi|^{-1+\sigma},\\
		\notag&\left|\partial_{\xi}^\alpha\frac{-(\mu_{\mathrm{I}}^2+\mu_{\mathrm{R}}^2)\, \mathrm{e}^{\lambda_1t}\,|\xi|^\sigma}{2\mu_{\mathrm{R}}\lambda_1-\mu_{\mathrm{I}}^2-\mu_{\mathrm{R}}^2-\lambda_1^2}\right|\lesssim \mathrm{e}^{-ct}\,|\xi|^{\sigma-|\alpha|},\\
		\notag&\left|\partial_{\xi}^\alpha\frac{(2\mu_{\mathrm{R}}\lambda_1-\lambda_1^2)\cos(\mu_{\mathrm{I}}t )\, \mathrm{e}^{\mu_{\mathrm{R}} t}\,|\xi|^\sigma}{2\mu_{\mathrm{R}}\lambda_1-\mu_{\mathrm{I}}^2-\mu_{\mathrm{R}}^2-\lambda_1^2}\right|\lesssim \mathrm{e}^{-ct}\,|\xi|^{-2+\sigma},\\
		\notag&\left|\partial_{\xi}^\alpha\frac{\lambda_1(\mu_{\mathrm{R}}\lambda_1+\mu_{\mathrm{I}}^2-\mu_{\mathrm{R}}^2)\sin(\mu_{\mathrm{I}}t )\, \mathrm{e}^{\mu_{\mathrm{R}} t}\,|\xi|^\sigma}{\mu_{\mathrm{I}}(2\mu_{\mathrm{R}}\lambda_1-\mu_{\mathrm{I}}^2-\mu_{\mathrm{R}}^2-\lambda_1^2)}\right|\lesssim \mathrm{e}^{-ct}\,|\xi|^{-1+\sigma}.
	\end{align}
\end{lemma}
\begin{proof}
	For the sake of briefness, we only show the proof of \eqref{lem-auxiliary-2-sin}, which is the most complicated one, and the other estimates can be justified analogously. Thanks to the asymptotic expansions of  characteristic roots for large frequencies, it is obvious that 
	\begin{align*}
		 |\partial_\xi^\alpha \lambda_1|\lesssim\begin{cases}
			1&\mbox{if}\ |\alpha|=0\\
			|\xi|^{-2-|\alpha|}&\mbox{if}\ |\alpha|\geqslant 1
		\end{cases},\quad |\partial_\xi^\alpha \mu_{\mathrm{R}}|\lesssim\begin{cases}
		1&\mbox{if}\ |\alpha|=0\\
		|\xi|^{-2-|\alpha|}&\mbox{if}\ |\alpha|\geqslant 1
	\end{cases},\quad 
	 |\partial_\xi^\alpha \mu_{\mathrm{I}}|\lesssim|\xi|^{1-|\alpha|}.
	\end{align*}
	Applying the derivative formula of compound functions in Lemma \ref{FadiBruno'sformula1}, one derives
	\begin{align*}
		|\partial_\xi^\alpha\mathrm{e}^{\mu_{\mathrm{R}}t}|&=\left|\sum_{k=1}^{|\alpha|}t^k\,\mathrm{e}^{\mu_{\mathrm{R}}t}\left(\sum_{\underset{|\gamma_1|+\cdots+|\gamma_k|=|\alpha|,\ |\gamma_j|\geqslant 1}{\gamma_1+\cdots+\gamma_k\leqslant \alpha}}(\partial_\xi^{\gamma_1}\mu_{\mathrm{R}})\,(\partial_\xi^{\gamma_2}\mu_{\mathrm{R}})\cdots(\partial_\xi^{\gamma_k}\mu_{\mathrm{R}})\right)\right|\\ 
		&\lesssim \sum_{k=1}^{|\alpha|}t^k\,\mathrm{e}^{\mu_{\mathrm{R}}t}\left(\sum_{\underset{|\gamma_1|+\cdots+|\gamma_k|=|\alpha|,\ |\gamma_j|\geqslant 1}{\gamma_1+\cdots+\gamma_k\leqslant \alpha}}|\xi|^{-(|\gamma_1|+\cdots+|\gamma_k|)-2k}\right)\\
		&\lesssim \sum_{k=1}^{|\alpha|} t^k\,\mathrm{e}^{-ct}\, |\xi|^{-|\alpha|-2k}\lesssim \mathrm{e}^{-ct}\, |\xi|^{-|\alpha|-2}
	\end{align*}
for $|\alpha|\geqslant 1$, and $|\partial_\xi^\alpha\mathrm{e}^{\mu_{\mathrm{R}}t}|\lesssim \mathrm{e}^{-ct}$ for $|\alpha|=0$, namely, $|\partial_\xi^\alpha\mathrm{e}^{\mu_{\mathrm{R}}t}|\lesssim \mathrm{e}^{-ct}\,|\xi|^{-|\alpha|}$ for any $|\alpha|\geqslant0$.
	By the same way, for any $|\alpha|\geqslant0$, we have
	\begin{align*}
		|\partial_\xi^\alpha \sin(\mu_\mathrm{I}t)|\lesssim 1+t^{|\alpha|},\quad  |\partial_\xi^\alpha |\xi|^\sigma|\lesssim |\xi|^{\sigma-|\alpha|},\quad \left|\partial_\xi^\alpha\big(\mu_{\mathrm{I}}(2\mu_{\mathrm{R}}\lambda_1-\mu_{\mathrm{I}}^2-\mu_{\mathrm{R}}^2-\lambda_1^2)\big)^{-1}\right|\lesssim|\xi|^{-3-|\alpha|}.
	\end{align*}
	Combining the above estimates, for any $|\alpha|\geqslant0$, one arrives at 
	\begin{align*}
		\mbox{LHS of \eqref{lem-auxiliary-2-sin}}&\lesssim \sum_{|\alpha_1|+\cdots+|\alpha_5|=|\alpha|}|\partial_\xi^{\alpha_1}(\mu_{\mathrm{R}}-\lambda_1)|\,|\partial_\xi^{\alpha_2}\sin(\mu_{\mathrm{I}}t)|\,|\partial_\xi^{\alpha_3}\mathrm{e}^{\mu_{\mathrm{R}}t}|\,|\partial_\xi^{\alpha_4}|\xi|^\sigma|\\ &\qquad\qquad\quad\qquad\times \left|\partial_\xi^{\alpha_5}\big(\mu_{\mathrm{I}}(2\mu_{\mathrm{R}}\lambda_1-\mu_{\mathrm{I}}^2-\mu_{\mathrm{R}}^2-\lambda_1^2)\big)^{-1}\right| \\ & \lesssim\sum_{|\alpha_1|+\cdots+|\alpha_5|=|\alpha|}|\xi|^{-|\alpha_1|}(1+t^{|\alpha_2|})\, \mathrm{e}^{-ct}\,|\xi|^{-|\alpha_3|}|\xi|^{\sigma-|\alpha_4|}|\xi|^{-3-|\alpha_5|}\\ &\lesssim \mathrm{e}^{-ct}\,|\xi|^{-3+\sigma},
	\end{align*}
	which demonstrates \eqref{lem-auxiliary-2-sin}. Although one may get better estimates when $|\alpha_1|\geqslant 1$ and $|\alpha_3|\geqslant 1$, the worse case occurs when $|\alpha_1|=|\alpha_3|=0$. It should be additionally mentioned that
	\begin{align*}
	\mbox{LHS of \eqref{lem-auxiliary-profile}}&\lesssim\sum\limits_{|\alpha_1|+|\alpha_2|+|\alpha_3|=|\alpha|}|\xi|^{\sigma-|\alpha_1|}\left(\sum\limits_{k=0}^{|\alpha_2|}t^k|\xi|^{k-|\alpha_2|}\right)\left(\sum\limits_{k=0}^{|\alpha_3|}t^k|\xi|^{2k-|\alpha_3|}\,\mathrm{e}^{-c|\xi|^2t}\right)\\
	&\lesssim \mathrm{e}^{-c|\xi|^2t}\, |\xi|^{\sigma-(|\alpha_1|+|\alpha_2|+|\alpha_3|)},
	\end{align*}
where we applied $t^k|\xi|^{2k}\,\mathrm{e}^{-c_0|\xi|^2t}\lesssim \mathrm{e}^{-c|\xi|^2t}$ and $t^k|\xi|^{k}\,\mathrm{e}^{-c_0|\xi|^2t}\lesssim t^{\frac{k}{2}}\,\mathrm{e}^{-c|\xi|^2t}$ to complete our proof.
\end{proof}

\begin{prop}\label{prop:L^q-L^q-large}
	Let $1<q<+\infty$ and $s\geqslant0$. The solution localizing in the large frequency zone to the dissipative MGT equation \eqref{Eq_MGT} with $\delta>0$ for $n\geqslant 2$ satisfies the following $L^q-L^q$ estimates: 
	\begin{align*}
		\|\varphi^{>0}(t,\cdot)\|_{H^s_{q,\chi_3}}\lesssim \mathrm{e}^{-ct}\|(\varphi_0^{>0},\varphi_1^{>0},\varphi_2^{>0})\|_{({H_q^s}\cap{H_q^{[s-1+s_0]_+}})\times H_q^{[s-1+s_0]_+}\times H_q^{[s-2+s_0]_+}},
	\end{align*}
	with the index $s_0:=n|\frac{1}{q}-\frac{1}{2}|$. Furthermore, the profile $\Psi(t,x)$ defined in Proposition \ref{prop:L^m-L^q-small} satisfies the following $L^q-L^q$ estimates: 
	\begin{align*}
		\|\Psi(t,\cdot)\|_{H^s_{q,\chi_3}}\lesssim \mathrm{e}^{-ct}\|(\varphi_1^{>0},\varphi_2^{>0})\|_{(H_q^{[s-1]_+})^2}.
	\end{align*}
\end{prop}

	\begin{proof}
		Recalling the representation formula 
		\begin{align*}
			|D|^s\varphi^{>0}(t,x)=\sum_{\ell=0,1,2}\big(|D|^sK_\ell^1(t,|D|)+|D|^sK_\ell^{\cos}(t,|D|)+|D|^sK_\ell^{\sin}(t,|D|)\big)\varphi^{>0}_\ell(x),	
		\end{align*}
		we only show the proof for the third data when $\ell=2$, and other proofs for the first two data when $\ell=0,1$ can be done analogously. It first holds that
		\begin{align*}
			 \chi_3(|D|)|D|^sK_2^1(t,|D|)\varphi_2^{>0}(x)=\ml{F}^{-1}_{\xi\to x}\left(\frac{-\chi_3(|\xi|)\,\mathrm{e}^{\lambda_1 t}\,|\xi|^{\min\{2,s\}}}{2\mu_{\mathrm{R}}\lambda_1-\mu_{\mathrm{I}}^2-\mu_{\mathrm{R}}^2-\lambda_1^2}\,|\xi|^{s-\min\{2,s\}}\widehat{\varphi}_2^{>0}(\xi)\right).
		\end{align*}
		By using \eqref{lem-auxiliary-2-1} with $\sigma=\min\{2,s\}$ so that $-2+\sigma\leqslant0$ and the variant of 
		Mikhlin-H\"omander multiplier theorem in Lemma \ref{Minklin-Homander-Theorem} with $b=0$, one arrives at 
		\begin{align*}
			\|\,|D|^s K_2^1(t,|D|)\varphi_2^{>0}(\cdot)\|_{L^q_{\chi_3}}\lesssim  \mathrm{e}^{-ct}\|\varphi_2^{>0}\|_{H_q^{[s-2]_+}}.
		\end{align*}
		Analogously,
		\begin{align*}
			&\chi_3(|D|)|D|^s K_2^{\cos}(t,|D|)\varphi_2^{>0}(x)\\
			&=\ml{F}^{-1}_{\xi\to x}\left(\frac{\chi_3(|\xi|)\cos(\mu_{\mathrm{I}}t)\, \mathrm{e}^{\mu_{\mathrm{R}} t}\,|\xi|^{\min\{s,2-s_0\}}}{2\mu_{\mathrm{R}}\lambda_1-\mu_{\mathrm{I}}^2-\mu_{\mathrm{R}}^2-\lambda_1^2}\,|\xi|^{s-\min\{s,2-s_0\}}\widehat{\varphi}_2^{>0}(\xi)\right),
		\end{align*}
	as well as			
		\begin{align*}		
			&\chi_3(|D|)|D|^s K_2^{\sin}(t,|D|)\varphi_2^{>0}(x)\\
			&=\ml{F}^{-1}_{\xi\to x}\left(\frac{\chi_3(|\xi|)(\lambda_1-\mu_{\mathrm{R}})\sin(\mu_{\mathrm{I}}t)\, \mathrm{e}^{\mu_{\mathrm{R}} t}\,|\xi|^{\min\{s,3-s_0\}}}{\mu_{\mathrm{I}}(2\mu_{\mathrm{R}}\lambda_1-\mu_{\mathrm{I}}^2-\mu_{\mathrm{R}}^2-\lambda_1^2)}\,|\xi|^{s-\min\{s,3-s_0\}}\widehat{\varphi}_2^{>0}(\xi)\right).
		\end{align*}
		By choosing $\sigma=\min\{s,2-s_0\}$ in \eqref{lem-auxiliary-2-cos} and $\sigma=\min\{s,3-s_0\}$ in \eqref{lem-auxiliary-2-sin}, an application of Lemma \ref{Minklin-Homander-Theorem} with $b=1$ shows 
		\begin{align*}
			\|\,|D|^s K_2^{\cos}(t,|D|)\varphi_2^{>0}(\cdot)\|_{L^q_{\chi_3}}\lesssim  \mathrm{e}^{-ct}\|\varphi_2^{>0}\|_{H_q^{[s-2+s_0]_+}},\\
			\|\,|D|^s K_2^{\sin}(t,|D|)\varphi_2^{>0}(\cdot)\|_{L^q_{\chi_3}}\lesssim  \mathrm{e}^{-ct}\|\varphi_2^{>0}\|_{H_q^{[s-3+s_0]_+}}.
		\end{align*}
		The summary of last estimates yields
		\begin{align*}
			\|\,|D|^s K_2(t,|D|)\varphi_2^{>0}(\cdot)\|_{L^q_{\chi_3}}\lesssim  \mathrm{e}^{-ct}\|\varphi_2^{>0}\|_{H_q^{[s-2+s_0]_+}}.
		\end{align*}
For another, because 
		\begin{align*}
			\chi_3(|D|)|D|^s\Psi(t,x)=\ml{F}^{-1}_{\xi\to x}\left(\chi_3(|\xi|)|\xi|^{\min\{s-1,0\}}\sin(|\xi|t)\, \mathrm{e}^{-\frac{\delta}{2}|\xi|^2t}\,|\xi|^{s-1-\min\{s-1,0\}}\big(\widehat{\varphi}_1^{>0}(\xi)+\tau\widehat{\varphi}_2^{>0}(\xi)\big)\right),
		\end{align*}
		we can complete the proof by choosing $\sigma=\min\{s-1,0\}$ in \eqref{lem-auxiliary-profile}.
	\end{proof}

	\subsection{Refined $L^p-L^q$ estimates}\label{Sub-Section-Smuumary}
	\hspace{5mm} Let us summarize all derived estimates in Propositions \ref{prop:L^m-L^q-middle}, \ref{prop:L^m-L^q-small} and \ref{prop:L^q-L^q-large} to get some refined estimates for the dissipative MGT equation with initial data belonging to
	\begin{align*}
	\ml{D}_{p,q}^s&:=({H_q^s}\cap{H_q^{[s-1+s_0]_+}}\cap L^p)\times (H_q^{[s-1+s_0]_+}\cap L^p)\times (H_q^{[s-2+s_0]_+}\cap L^p),\\
	\widetilde{\ml{D}}_{p,q}^s&:=({H_q^s}\cap{H_q^{[s-1+s_0]_+}}\cap L^p)\times (H_q^{[s-1+s_0]_+}\cap L^p)\times (H_q^{[s-2+s_0]_+}\cap H_q^{[s-1]_+}\cap L^p),
	\end{align*}
with $s_0=n|\frac{1}{q}-\frac{1}{2}|$ and $s\geqslant0$.

	\begin{theorem}\label{thm:L^m-L^q-estimate}
Let $1\leqslant p\leqslant q< +\infty$, $q\neq1$ and $s\geqslant0$.	The solution to the dissipative MGT equation \eqref{Eq_MGT} with $\delta>0$ for $n\geqslant 2$ satisfies the following $( L^q\cap L^p)-L^q$ estimates:
\begin{align*}
	\|\varphi^{>0}(t,\cdot)\|_{\dot{H}^{s}_{q}}\lesssim \begin{cases}
		(1+t)^{-(1+\frac{n}{2}+\frac{1}{2}\lfloor\frac{n}{2}\rfloor)(\frac{1}{p}-\frac{1}{q})+\frac{3}{2}+\frac{1}{2}\lfloor\frac{n}{2}\rfloor-\frac{s}{2}}\|(\varphi_0^{>0},\varphi_1^{>0},\varphi_2^{>0})\|_{\ml{D}_{p,q}^s}&\mbox{if}\ s\in[0,1)\cup(1,2],\\
		(1+t)^{-(\frac{1}{2}+\frac{n}{2}+\frac{1}{2}\lfloor\frac{n}{2}\rfloor)(\frac{1}{p}-\frac{1}{q})+\frac{1}{2}+\frac{1}{2}\lfloor\frac{n}{2}\rfloor}\|(\varphi_0^{>0},\varphi_1^{>0},\varphi_2^{>0})\|_{\ml{D}_{p,q}^s}&\mbox{if}\ s=1,\\			
		(1+t)^{-\frac{3n}{4}(\frac{1}{p}-\frac{1}{q})+\frac{1}{2}+\frac{n}{4}-\frac{s}{2}}\|(\varphi_0^{>0},\varphi_1^{>0},\varphi_2^{>0})\|_{\ml{D}_{p,q}^s}&\mbox{if}\ s\in(2,+\infty).
	\end{cases}
\end{align*}
Furthermore, by subtracting the profile $\Psi(t,x)$, the following refined estimates hold:
\begin{align*}
	&\| (\varphi^{>0}-\Psi)(t,\cdot)\|_{\dot{H}^{s}_{q}}\lesssim\begin{cases}
		(1+t)^{-(\frac{1}{2}+\frac{n}{2}+\frac{1}{2}\lfloor\frac{n}{2}\rfloor)(\frac{1}{p}-\frac{1}{q})+\frac{1}{2}+\frac{1}{2}\lfloor\frac{n}{2}\rfloor}\|(\varphi_0^{>0},\varphi_1^{>0},\varphi_2^{>0})\|_{\widetilde{\ml{D}}_{p,q}^s}&\mbox{if}\ s=0,\\
		(1+t)^{-(1+\frac{n}{2}+\frac{1}{2}\lfloor\frac{n}{2}\rfloor)(\frac{1}{p}-\frac{1}{q})+1+\frac{1}{2}\lfloor\frac{n}{2}\rfloor-\frac{s}{2}}\|(\varphi_0^{>0},\varphi_1^{>0},\varphi_2^{>0})\|_{\widetilde{\ml{D}}_{p,q}^s}&\mbox{if}\ s\in(0,1],\\
		(1+t)^{-\frac{3n}{4}(\frac{1}{p}-\frac{1}{q})+\frac{n}{4}-\frac{s}{2}}\|(\varphi_0^{>0},\varphi_1^{>0},\varphi_2^{>0})\|_{\widetilde{\ml{D}}_{p,q}^s}&\mbox{if}\ s\in(1,+\infty).
	\end{cases}
\end{align*}
	\end{theorem}
	\begin{remark}
	By subtracting the function $\Psi(t,\cdot)$ in the $L^q$ norm, we always notice that the derived estimates for $\varphi^{>0}(t,\cdot)$ can be improved.  For this reason, we may explain it as the asymptotic profile of solution in the $L^q$ framework. Note that this profile $\Psi(t,x)$ is exactly the same as those in the $L^2$ setting (cf. \cite{Chen-Ikehata=2021,Chen-Takeda=2023}). In other words, the singular diffusion waves function equipped the combined data plays an essential role in the dissipative MGT equation.
\end{remark}
\begin{exam}
The acoustic velocity potential satisfies
\begin{align*}
\|\varphi^{>0}(t,\cdot)\|_{L^q}&\lesssim (1+t)^{-(1+\frac{n}{2}+\frac{1}{2}\lfloor\frac{n}{2}\rfloor)(\frac{1}{p}-\frac{1}{q})+\frac{3}{2}+\frac{1}{2}\lfloor\frac{n}{2}\rfloor}\|(\varphi_0^{>0},\varphi_1^{>0},\varphi_2^{>0})\|_{\ml{D}_{p,q}^0}.
\end{align*}
Furthermore, by subtracting its profile $\Psi(t,x)$, there is the improvement $(1+t)^{\frac{1}{2}(\frac{1}{p}-\frac{1}{q})-1}$ which is decay, precisely,
\begin{align*}
\| (\varphi^{>0}-\Psi)(t,\cdot)\|_{L^q}\lesssim (1+t)^{-(\frac{1}{2}+\frac{n}{2}+\frac{1}{2}\lfloor\frac{n}{2}\rfloor)(\frac{1}{p}-\frac{1}{q})+\frac{1}{2}+\frac{1}{2}\lfloor\frac{n}{2}\rfloor}\|(\varphi_0^{>0},\varphi_1^{>0},\varphi_2^{>0})\|_{\widetilde{\ml{D}}_{p,q}^0}.
\end{align*}
\end{exam}
\begin{exam}
	The gradient of acoustic velocity potential (i.e. the acoustic velocity $\mathbf{u}^{>0}=-\nabla\varphi^{>0}$ in irrotational flows) satisfies
	\begin{align*}
		\|\nabla \varphi^{>0}(t,\cdot)\|_{L^q}&\lesssim (1+t)^{-(\frac{1}{2}+\frac{n}{2}+\frac{1}{2}\lfloor\frac{n}{2}\rfloor)(\frac{1}{p}-\frac{1}{q})+\frac{1}{2}+\frac{1}{2}\lfloor\frac{n}{2}\rfloor}\|(\varphi_0^{>0},\varphi_1^{>0},\varphi_2^{>0})\|_{\ml{D}_{p,q}^1}.
	\end{align*}
It has a faster decay rate than the one of acoustic velocity potential itself carrying the improvement $(1+t)^{\frac{1}{2}(\frac{1}{p}-\frac{1}{q})-1}$, which is the so-called parabolic effect.
	Furthermore, by subtracting its profile $\Psi(t,x)$, there is the improvement $(1+t)^{-\frac{1}{2}(\frac{1}{p}-\frac{1}{q})}$ which is decay if $p< q$, precisely,
	\begin{align*}
		\| \nabla (\varphi^{>0}-\Psi)(t,\cdot)\|_{L^q}\lesssim (1+t)^{-(1+\frac{n}{2}+\frac{1}{2}\lfloor\frac{n}{2}\rfloor)(\frac{1}{p}-\frac{1}{q})+\frac{1}{2}+\frac{1}{2}\lfloor\frac{n}{2}\rfloor}\|(\varphi_0^{>0},\varphi_1^{>0},\varphi_2^{>0})\|_{\widetilde{\ml{D}}_{p,q}^1}.
	\end{align*}
Note that in the above we employed the equivalence of norms $\|\nabla f\|_{L^q}\approx \|\,|D|f\|_{L^q}$ due to $1<q<+\infty$. One may see \cite[Lemma 3.6]{Palmieri-Reissig=2018}.
\end{exam}

\section{Conservative MGT equation in the $L^q$ framework}\label{Section_inviscid-JMGT}\setcounter{equation}{0}
\hspace{5mm}Let us turn to the conservative MGT equation \eqref{Eq_MGT} with $\delta=0$. By introducing the good unknown $u=u(t,x)$ such that
\begin{align}\label{unknown}
u(t,x):=\tau\varphi_t^{=0}(t,x)+\varphi^{=0}(t,x),
\end{align}
it satisfies the free wave equation as follows:
\begin{align}\label{Eq-01}
\begin{cases}
u_{tt}-\Delta u=0,&x\in\mb{R}^n,\ t>0,\\
u(0,x)=\tau\varphi_1^{=0}(x)+\varphi_0^{=0}(x),\ u_t(0,x)=\tau\varphi_2^{=0}(x)+\varphi_1^{=0}(x),&x\in\mb{R}^n.
\end{cases}
\end{align}
We are going to derive some $L^p-L^q$ estimates for $\varphi^{=0}(t,\cdot)$ according to the well-established results for $u(t,\cdot)$ in \cite{RSS-1970,JCP-1980} and references therein.

Before stating our main result, let us take the closed triangle $\triangle_{P_1P_2P_3}$ with the vertices
\begin{align*}
P_1=\left(\frac{1}{2}+\frac{1}{n+1},\frac{1}{2}-\frac{1}{n+1}\right),\ P_2=\left(\frac{1}{2}-\frac{1}{n-1},\frac{1}{2}-\frac{1}{n-1}\right),\ P_3=\left(\frac{1}{2}+\frac{1}{n-1},\frac{1}{2}+\frac{1}{n-1}\right)
\end{align*}
for $n\geqslant 3$, and $P_2=(0,0)$, $P_3=(1,1)$ for $n=1,2$.
\begin{theorem}\label{Thm-Conser-MGT}
Let $(\frac{1}{p},\frac{1}{q})\in\triangle_{P_1P_2P_3}$ and $s\geqslant0$. The solution to the conservative MGT equation \eqref{Eq_MGT} with $\delta=0$ satisfies the following $L^p-L^q$ estimates:
\begin{align*}
\|\varphi^{=0}(t,\cdot)\|_{\dot{H}^s_q}\lesssim t^{1-n(\frac{1}{p}-\frac{1}{q})}\|(\varphi_0^{=0},\varphi_1^{=0},\varphi_2^{=0})\|_{(H^{s+1}_p\cap H^s_q)\times H^{s+1}_p\times H^{s}_p}
\end{align*}
for $t>0$.
\end{theorem}
\begin{remark}
If one interests in the estimate for any $t\geqslant0$, particularly an use in some nonlinear problems, we just need to restrict the pair $(\frac{1}{p},\frac{1}{q})$ belonging to the trapezoid with the vertices $P_2$, $P_3$ and
\begin{align*}
P_4=\left(\frac{1}{2}+\frac{1}{n},\frac{1}{2}\right),\ P_5=\left(\frac{1}{2},\frac{1}{2}-\frac{1}{n}\right),
\end{align*}
in which $1-n(\frac{1}{p}-\frac{1}{q})\geqslant0$ holds.
\end{remark}
\begin{proof}
We actually know from \eqref{unknown} that
\begin{align*}
\begin{cases}
\tau\varphi_t^{=0}+\varphi^{=0}=u,&x\in\mb{R}^n,\ t>0,\\
\varphi^{=0}(0,x)=\varphi_0^{=0}(x),&x\in\mb{R}^n,
\end{cases}
\end{align*}
whose solution is given via the Duhamel principle via
\begin{align*}
\varphi^{=0}(t,x)=\frac{1}{\tau}\int_0^t\mathrm{e}^{-\frac{1}{\tau}(t-\eta)}\,u(\eta,x)\,\mathrm{d}\eta+\mathrm{e}^{-\frac{1}{\tau}t}\,\varphi_0^{=0}(x).
\end{align*}
An application of the Minkowski inequality (in the integral form) leads to
\begin{align*}
\left\|\int_0^t\mathrm{e}^{-\frac{1}{\tau}(t-\eta)}\,u(\eta,\cdot)\,\mathrm{d}\eta\right\|_{L^q}&=\left(\int_{\mb{R}^n}\left|\int_0^t\mathrm{e}^{-\frac{1}{\tau}(t-\eta)}\,u(\eta,x)\,\mathrm{d}\eta\right|^q\mathrm{d}x\right)^{\frac{1}{q}}\\
&\lesssim\int_0^t\left(\int_{\mb{R}^n}\left|\mathrm{e}^{-\frac{1}{\tau}(t-\eta)}\,u(\eta,x)\right|^q\mathrm{d}x\right)^{\frac{1}{q}}\mathrm{d}\eta\\
&\lesssim\int_0^t\mathrm{e}^{-\frac{1}{\tau}(t-\eta)}\|u(\eta,\cdot)\|_{L^q}\,\mathrm{d}\eta.
\end{align*}
The well-known results in \cite{RSS-1970,JCP-1980} show
\begin{align*}
\|u(t,\cdot)\|_{L^q}\lesssim t^{1-n(\frac{1}{p}-\frac{1}{q})}\left(\|u(0,\cdot)\|_{H^1_p}+\|u_t(0,\cdot)\|_{L^p}\right)
\end{align*}
for any $(\frac{1}{p},\frac{1}{q})\in\triangle_{P_1P_2P_3}$. Therefore, we next study
\begin{align*}
\ml{I}(t):=\int_0^t\mathrm{e}^{-\frac{1}{\tau}(t-\eta)}\,\eta^{1-n(\frac{1}{p}-\frac{1}{q})}\,\mathrm{d}\eta.
\end{align*}
By separating the interval $[0,t]$ into $[0,\frac{t}{2}]$ and $[\frac{t}{2},t]$, one obtains
\begin{align*}
\ml{I}(t)&\lesssim\mathrm{e}^{-\frac{1}{2\tau}t}\int_0^{\frac{t}{2}}\eta^{1-n(\frac{1}{p}-\frac{1}{q})}\,\mathrm{d}\eta+t^{1-n(\frac{1}{p}-\frac{1}{q})}\int_{\frac{t}{2}}^t\mathrm{e}^{-\frac{1}{\tau}(t-\eta)}\,\mathrm{d}\eta\\
&\lesssim\mathrm{e}^{-\frac{1}{2\tau}t}\,t^{2-n(\frac{1}{p}-\frac{1}{q})}+\big(1-\mathrm{e}^{-\frac{1}{2\tau}t}\big)\,t^{1-n(\frac{1}{p}-\frac{1}{q})}\\
&\lesssim t^{1-n(\frac{1}{p}-\frac{1}{q})}
\end{align*}
due to $1-n(\frac{1}{p}-\frac{1}{q})>-1$ for any $(\frac{1}{p},\frac{1}{q})\in\triangle_{P_1P_2P_3}$. As a consequence, recalling the initial data in \eqref{Eq-01} and applying all derived estimates in the above, we may deduce
\begin{align*}
\|\varphi^{=0}(t,\cdot)\|_{\dot{H}^s_q}&\lesssim\int_0^t\mathrm{e}^{-\frac{1}{\tau}(t-\eta)}\|u(\eta,\cdot)\|_{\dot{H}^s_q}\,\mathrm{d}\eta+\mathrm{e}^{-\frac{1}{\tau}t}\|\varphi_0^{=0}\|_{\dot{H}^s_q}\\
&\lesssim\ml{I}(t)\left(\|\tau\varphi_1^{=0}+\varphi_0^{=0}\|_{H^{s+1}_p}+\|\tau\varphi_2^{=0}+\varphi_1^{=0}\|_{H^s_p}\right)+\mathrm{e}^{-\frac{1}{\tau}t}\|\varphi_0^{=0}\|_{\dot{H}^s_q}\\
&\lesssim t^{1-n(\frac{1}{p}-\frac{1}{q})}\|(\varphi_0^{=0},\varphi_1^{=0},\varphi_2^{=0})\|_{(H^{s+1}_p\cap H^s_q)\times H^{s+1}_p\times H^{s}_p}.
\end{align*}
Our proof is complete.
\end{proof}
	
\section{Final remarks}\label{Section_Final_Remark}\setcounter{equation}{0}
\hspace{5mm}We in the present paper have derived some $L^p-L^q$ estimates for the dissipative MGT equation in Theorem \ref{thm:L^m-L^q-estimate}
\begin{center}
	with the parabolic-like decay rate $-\frac{n}{2}(\frac{1}{p}-\frac{1}{q})$ among other factors,
\end{center} 
and for the conservative MGT equation in Theorem \ref{Thm-Conser-MGT}
\begin{center}
with the hyperbolic-like decay rate $-n(\frac{1}{p}-\frac{1}{q})$ among other factors.
\end{center}   This is mainly caused by the dissipation $-\delta\Delta\varphi_t$ in the MGT equation \eqref{Eq_MGT}. For another, due to the parabolic-like structure, the admissible range of $(\frac{1}{p},\frac{1}{q})$ in the dissipative case is larger than the one in the conservative case (the wave-like structure leads to the closed triangle range only).

Thanks to the derived $L^p-L^q$ estimates, we expect that they can be useful to study global (in time) behavior of solutions, including existence and asymptotic profiles, for some nonlinear MGT equations in the whole space $\mb{R}^n$. For example, the semilinear MGT equations with power nonlinearities $|\varphi|^m$ or $|\varphi_t|^m$ can be studied by constructing suitable time-weighted Sobolev spaces with time-dependent functions from our main results, i.e. Theorems \ref{thm:L^m-L^q-estimate} and \ref{Thm-Conser-MGT}. 

	\appendix
	\section{Tools from the Fourier analysis}\label{Appendix-A}
	\hspace{5mm}We present some tools from the Fourier analysis that have been applied in the dissipative MGT equation in the general $L^q$ framework.
	\begin{lemma}[Fourier Transform via Modified Bessel Function, \cite{Grafakos=2004}] \label{Modified-Bessel-Funct}
		Let $f(x)=f_0(|x|)$ be a radial function defined on $\mb{R}^n$, where $f_0$ is defined on $[0,+\infty)$. Then, the Fourier transform of $f$ is given by the formula
		\begin{align*}
			\widehat{f}(|\xi|)= c \int_0^{+\infty} f_0(r)\, r^{n-1} \widetilde{\mathcal{J}}_{\frac{n}{2}-1}(r|\xi|)\,\mathrm{d}r,
		\end{align*}
		where $\widetilde{\mathcal{J}}_\mu(s):=s^{-\mu}\mathcal{J}_\mu(s)$ is the so-called modified Bessel function with the classical Bessel function $\ml{J}_\mu(s)$ and $\mu\in\mb{N}_0$. It  holds analogously for the inverse Fourier transform.
	\end{lemma}
	
	\begin{lemma}[Bernstein Theorem, \cite{Sj=1970}]\label{Bernstein-Theorem} Let $\widehat{f}\in H^N$ with $N>\frac{n}{2}$ and $n\geqslant 1$. Then, $\ml{F}^{-1}(\widehat{f})\in L^1$ and satisfies
		\begin{align*}
		\|\ml{F}^{-1}(\widehat{f})\|_{L^1}\lesssim\|\widehat{f}\|_{L^2}^{1-\frac{n}{2N}}\left(\sum\limits_{|\alpha|=N}\|\partial_{\xi}^{\alpha}\widehat{f}\|_{L^2}\right)^{\frac{n}{2N}}.
		\end{align*}
	\end{lemma}
	
	\begin{lemma}[Variant of Mikhlin-H\"omander Multiplier Theorem, \cite{Miyachi=1980}]\label{Minklin-Homander-Theorem}
		Let $1<q<+\infty$, $k=[\frac{n}{2}]+1$ and $b\geqslant0$. Let $\ml{M}\in\ml{C}^k(\mb{R}^n\backslash\{0\})$  such that $\ml{M}(\xi)=0$ if $|\xi|\leqslant 1$ and 
		\begin{align*}
			|\partial_{\xi}^{\alpha}\ml{M}(\xi)|\lesssim |\xi|^{-nb|\frac{1}{q}-\frac{1}{2}|}\left(A|\xi|^{b-1}\right)^{|\alpha|}
		\end{align*}
		for any $|\alpha|\leqslant k$ and $|\xi|\geqslant 1$, with a constant $A\geqslant 1$. Then, the operator $\ml{T}_{\ml{M}}:=\ml{F}^{-1}_{\xi\to x}(\ml{M}(t,\xi))\ast_{(x)}$ with a parameter $t$, defined by the action $(\ml{T}_{\ml{M}}f)(t,\cdot)=\ml{F}_{\xi\to x}^{-1}(\ml{M}(t,\xi)\widehat{f}(\xi))$, is continuously bounded from $L^q$ into itself satisfying
		\begin{align*}
			\|(\ml{T}_{\ml{M}}f)(t,\cdot)\|_{L^q}\lesssim A^{n|\frac{1}{q}-\frac{1}{2}|}\|f\|_{L^q}.
		\end{align*}
	\end{lemma}
	
	\begin{lemma}[Derivative of Compound Function, \cite{Simader=1972}] \label{FadiBruno'sformula1}
		The formula for higher-order derivatives of compound function $(f\circ g)(x)=f(g(x))$ holds
		\begin{align*}
			\partial_x^{\alpha}(f\circ g)=\sum\limits_{k=1}^{|\alpha|}(f^{(k)}\circ g)\left(\sum\limits_{\substack{\gamma_1+\cdots+\gamma_{k}\leqslant \alpha\\|\gamma_1|+\cdots+|\gamma_k|=|\alpha|,\ |\gamma_j|\geqslant 1}}\big(\partial_x^{\gamma_1}g(x)\big)\cdots\big(\partial_x^{\gamma_k}g(x)\big)\right)
		\end{align*}
		for any multi-index $\alpha$.
	\end{lemma}

	\section*{Acknowledgments} 
	Wenhui Chen is supported in part by the National Natural Science Foundation of China (grant No. 12301270, grant No. 12171317), Guangdong Basic and Applied Basic Research Foundation (grant No. 2025A, grant No. 2023A1515012044), 2024 Basic and Applied Basic Research Topic--Young Doctor Set Sail Project (grant No. 2024A04J0016). The authors thank Giovanni Girardi (Polytechnic University of Marche) for some discussions on Lemma \ref{lem-oscillating-sincos}.

\end{document}